\documentclass[11pt,a4paper]{amsart}
 \usepackage[]{amsmath, amsthm, amsfonts, amssymb, amscd,mathtools}
\usepackage[mathscr]{eucal}
\usepackage{hyperref}
\usepackage[all]{xy}
\usepackage[usenames,dvipsnames]{color}
\usepackage{todonotes}\usepackage{rotating}
\usepackage[shortlabels]{enumitem}
\usepackage{tikz}

\setlist[enumerate]{label=\it{(\roman*)},
  ref=\it{(\roman*)}}

\usepackage{dsfont}

\usepackage{blindtext}
\usepackage{subfiles}

\usepackage{stackengine}
\stackMath

\newcommand{\A}{\mathds{A}}

\newcommand{\C}{\mathds{C}}

\renewcommand{\P}{\mathds{P}}
\newcommand{\Q}{\mathds{Q}}
\newcommand{\R}{\mathds{R}}

\newcommand{\Z}{\mathds{Z}}

\newcommand{\caA}{\mathcal{A}}

\newcommand{\caO}{\mathcal{O}}

\newcommand{\caR}{\mathcal{R}}

\newcommand{\caZ}{\mathcal{Z}}

\newcommand{\sI}{\mathscr{I}}

\newcommand{\sL}{\mathscr{L}}

\newcommand{\sO}{\mathscr{O}}

\newcommand{\fg}{{\mathfrak{g}}}

 \DeclareMathOperator {\Div}{div}
 
 \DeclareMathOperator{\im}{Im}
 \DeclareMathOperator {\cl}{cl}
 \DeclareMathOperator {\CH}{CH} 
  
\DeclareMathOperator{\Spec}{Spec}

\numberwithin{equation}{section}

\theoremstyle{plain}

\newtheorem{prop}{Proposition}[section]

\newtheorem{cor}[prop]{Corollary}
\newtheorem{lem}[prop]{Lemma}
\newtheorem{thm}[prop]{Theorem}
\newtheorem{theoalph}{Theorem}

\theoremstyle{definition}
\newtheorem{df}[prop]{Definition}

\newtheorem{defalph}{Definition}

\theoremstyle{remark}
\newtheorem{rmk}[prop]{Remark}
\newtheorem{ex}[prop]{Example}

\newcommand{\wt}{\widetilde}
\newcommand{\wh}{\widehat}
\newcommand{\Pic}{{\rm Pic}}
\newcommand{\cyc}{{\rm cyc}}

\newcommand{\surj}{\twoheadrightarrow}
\newcommand{\inj}{\hookrightarrow}

\newcommand{\id}{{\operatorname{id}}}
\newcommand{\un}{\underline}
\newcommand{\ov}{\overline}

\begin{document}

\renewcommand\stackalignment{l}

\title{Arithmetic Cycles with Modulus}

\author{Souvik~Goswami}
\address{3 Beni Master Lane, Kolkata, WB, 700061, India.}
\email{gossouvik@gmail.com}

\author{Rahul Gupta}
\address{Institute of Mathematical Sciences, A CI of Homi Bhabha National Institute, 4th Cross St., CIT Campus, Tharamani, Chennai,
  600113, India.} 
  \email{rahulgupta@imsc.res.in}

\date{\today}

\thanks{Souvik Goswami was supported by an SFB postdoctoral fellowship at the University of Regensburg, the University of Barcelona Mar\'ia Zambrano postdoctoral fellowship, and partially funded by the Spanish MICINN research project PID2023-147642NB-I00 during a large duration of the project.}

\subjclass{14G40, 14C22, 14C25}
\keywords{Arithmetic Chow groups, Cycles with modulus, relative cohomology groups.}
\date{\today}
\newif\ifprivate
\privatetrue

\begin{abstract}
We add analytic components to algebraic cycles with modulus and construct an arithmetic Chow group with modulus that resembles the classical arithmetic Chow groups by Gillet and Soul\'e. The analytic component is dictated by imposing a vanishing condition on the cohomology class of a cycle with modulus. We establish several natural properties of this group as a consequence. 
\end{abstract}

\maketitle

\setcounter{tocdepth}{1}

\tableofcontents{}

\section{Introduction}\label{sec:Intro}
\subsection{Motivation}\label{sec:Motiv}
Let $k$ be a subfield of $\C$. Given a smooth quasi-projective variety $X$ over $k$, its Chow ring $\bigoplus_{p\geq 0}\CH^{p}(X)$ gives a nice notion of geometrically defined cohomology theory with many desirable properties. The chief among them is the (rational) isomorphism 
\begin{equation} \label{eqn:intro-1}
    K_{0}(X)\otimes \Q\cong \bigoplus_{p\geq 0}\CH^{p}(X)\otimes \Q
\end{equation}
with the Grothendieck group of coherent locally-free sheaves, given by the Chern character map. Recall that the first Chern class gives an isomorphism $c_{1}\colon \Pic(X)\cong \CH^{1}(X)$ between the group of isomorphism classes of line bundles and the divisor class group on $X$, which is the initial motivation for isomorphism \eqref{eqn:intro-1}. For a more general singular variety $X$, \eqref{eqn:intro-1} does not hold. Also, several well-known properties associated with smooth projective varieties are no longer true if one drops the projectiveness condition, for example, the existence of a degree map and the existence of the Albanese variety of $X$.  
%One of the key reasons is that the group $\CH^p(X)$ does not capture anything related to the boundaries $D=\overline{X}\setminus X$ of smooth compactifications $X\subset \overline{X}$.
One way to resolve this is to take a smooth compactification $X\subset \overline{X}$ such that $D=\overline{X}\setminus X$ is supported on an effective Caritier divisor, and to consider cycles with certain ramification conditions along the boundary $D$. 
These observations led to the quest for a geometrically defined cohomology theory that is sensitive to a divisor of a general smooth quasi-projective variety $X$  
%Here, by 
%a compactification $X\subset \overline{X}$, 
and also restores \eqref{eqn:intro-1} in a suitable setting. 
%Here $D=\overline{X}\setminus X$ with a smooth compactification $X\subset \overline{X}$.

Chow groups with modulus were defined in part to address this quest. The very first notion in this direction was given by Bloch and Esnault in \cite{BE03}, and over the years significant development was made in \cite{KL-08}, \cite{Park09}, \cite{KP12}, \cite{KP17}, \cite{Kerz-Saito}, \cite{Binda-Saito}, \cite{GK22}, et al. Following \cite[\S3]{Binda-Saito}, for a smooth quasi-projective variety $X$ over $k$ and an effective divisor $D$, one defines the Chow groups with modulus at $D$ heuristically as follows: Let $\caZ^{p}(X|D)$ be the free abelian group generated by reduced and irreducible closed subvarieties $Z$ of $X$ that do not meet $D$. Any such $Z$ will be called a (codimension $p$) cycle with modulus. Furthermore, let $\caR^{p}(X|D)$ be the subgroup defined by divisors of rational functions on subvarieties of codimension $p-1$, which are regular, in fact, a unit in a neighborhood of $D$, and are congruent to one modulo the ideal sheaf of $D$ (see Definition \ref{eqn:CH-M*-1-0} for details). We define the Chow group with modulus on $D$ as the quotient
\begin{equation} \label{eqn:intro-2}
\CH^{p}(X|D)\coloneqq \caZ^{p}(X|D)/\caR^{p}(X|D). 
\end{equation}
Although it is not clear whether the relative version of the isomorphism \eqref{eqn:intro-1} is true, it is known that the first Chern class map induces an isomorphism $\CH^{1}(X|D) \cong \Pic(X,D)$, where $\Pic(X, D)$ is the relative Picard group that consists of isomorphic classes of line bundles (that are trivial along $D$) with isomorphisms to $\caO_{D}$ on $D$. For $X$ proper, one can define a degree map on zero cycles with modulus (\cite[Proposition 2.10]{KP12}) and for $X$ projective, one can define an Albanese with modulus (\cite[Theorems 10.3 and 11.3]{BAK18}).

On the other end of the spectrum, arithmetic Chow groups were developed by Gillet and Soul\'e in \cite{GS90} to give a refined intersection theory for arithmetic varieties. An arithmetic variety $X$ is a scheme that is flat and of finite type over $\Spec(\Z)$, with smooth generic fibre $X_{\Q}$. If $X$ is regular, the notion of a geometric intersection theory was first established by Gillet and Soul\'e, using $K$-theory. However, without a notion of the completeness of such schemes, the idea of a degree associated with such an intersection theory was not readily available. For example, if $X$ is a proper arithmetic variety, then one can hope to define a degree with values in $\CH^{1}(\Spec(\Z))$. But this group is trivial, and would not be the right choice. The solution proposed by Gillet and Soul\'e was to compactify such arithmetic varieties by adding the fibre at infinity related to the embedding $\Q\hookrightarrow \C$. So, one can view a compactified arithmetic variety as a family over each prime, including the prime at infinity. Given a compactified arithmetic variety $X$ defined over $\Spec(\Z)$, one defines the arithmetic Chow group $\widehat{\CH}^{p}(X)$ of codimension $p$ as the free abelian group generated by tuples $(Z,g_{Z})$ modulo a suitable arithmetic rational equivalence, where $Z\in \caZ^{p}(X)$ is a cycle of codimension $p$, and $g_{Z}$ is an analytic object, called the Green current for the cycle $Z$, defined on the analytic space $X(\C)$ associated with $X$. With this definition, one can show that $\widehat{\CH}^{1}(\Spec(\Z))\cong \R$ and hence define an arithmetic degree morphism $\wh{\text{deg}}$ with values in $\R$. These arithmetic Chow groups come equipped with several desirable properties, namely flat pullback, proper pushforward, and, using $\wh{\text{deg}}$, a refined notion of an arithmetic intersection degree that has been used in questions related to arithmetic geometry. For example, Faltings defined the height of a closed irreducible subset of $\P^{n}_{\Z}$ using this intersection degree, which was crucial in the proof of the Mordell conjecture. 
 
The current article originated from the following natural question asked by the second author during a talk given by the first author: Assume that $X$ is a smooth and quasi-projective variety over a number field $k$, and $D$ is an effective divisor such that $D_{red}$ is a simple normal crossing divisor (sncd for short). Under these assumptions, how can one define $\wh{\CH}^{p}(X|D)$, an arithmetic Chow group with modulus? In this paper, we propose an answer to the above question by taking into account, for a cycle with modulus, the behaviour of its analytic fundamental class at $D_{red}(\C)$. We show that these groups have natural functorial properties with respect to appropriate morphisms, are modules over the usual arithmetic Chow groups, fit into a natural exact sequence, and the group of arithmetic divisors with modulus is isomorphic to the Hermitian relative Picard group. In synopsis, this provides a very natural amalgamation of Gillet and Soul\'e's theory with the theory of Chow group with modulus.

%Beyond the definition, we prove several properties of $\widehat{\CH}^{p}(X|D)$ that are analogous to those of $\wh{\CH}^{p}(X)$.
%Intuitively, these properties are natural once we have set up the definition of $\wh{\CH}^p(X|D)$. 
The results in this article can be viewed as a first step towards defining a concrete realization of arithmetic motivic cohomology with modulus, and giving an arithmetic analogue of the higher-dimensional global class-field theory. While this will be a long-term project, one of our short-term aims is to define an arithmetic analogue of the relative Grothendieck group $K_0(X, D)$, and study its connection with $\wh{\CH}^{p}(X|D)$. This will be the subject of our next paper.

\subsection{Results}
Let, as before, $X$ denote a smooth quasi-projective variety of dimension $d$, defined over a number field $k$. For ease of writing, we do not distinguish between $X$ and its corresponding analytic space $X(\C)$. Let $\caZ^{p}(X)$ be the free abelian group generated by reduced and irreducible subvarieties of codimension $p$ that compute $\CH^{p}(X)$. The definition of arithmetic Chow groups involves the following basic principle: Typically, the cycle class map
\begin{displaymath}
\cl_{p}\colon \caZ^{p}(X)\rightarrow H^{2p}(X) 
\end{displaymath}
can be represented by the fundamental class $\delta_{Z}$ of an algebraic cycle in homology, as well as a smooth form $\omega_{Z}$ of type $(p,p)$ in cohomology. Since they represent the same cycle class, the difference between the fundamental class and the current $[\omega_{Z}]$ associated to $\omega_{Z}$ (see Lemma \ref{lem:form--current-inc} for the definition) satisfies the differential equation 
\begin{displaymath}
\partial\overline{\partial}g_{Z}=[\omega_{Z}]-\delta_{Z}.
\end{displaymath}
The current $g_{Z}$ is called a Green current associated with the algebraic cycle $Z$. It can be chosen to be of type $(p-1,p-1)$ and satisfying $F^{\ast}_{\infty}g_{Z}=(-1)^{p-1}g_{Z}$ for the conjugate linear involution $F_{\infty}$ induced by the complex conjugation on $X(\C)$. The arithmetic Chow groups $\widehat{\CH}^{p}(X)$ are now defined by taking pairs $(Z,g_{Z})$ of algebraic cycles $Z$ and an associated Green current $g_{Z}$, and finally quotienting out by a suitable arithmetic rational equivalence. We note here that if $X$ is not projective, one has to invoke the notion of Green forms of logarithmic type along $|Z|$ to get a suitable Green current.

Now, suppose that $D$ is an effective divisor on $X$ whose support $E$ is an sncd. Following the above principle, if one wants to find a suitable theory of arithmetic Chow group with modulus, it needs a description of a cycle class map to a suitable cohomology group that is sensitive to the modulus $D$. Such a map was first introduced in \cite[\S 7,\S 8]{Binda-Saito} in a general context of higher Chow groups with modulus, and can be described in our setting as follows: As mentioned before, any $Z\in \caZ^{p}(X|D)$ is an ordinary algebraic cycle in $\caZ^{p}(X)$. The classical cycle class map can be seen as a composition
\begin{displaymath}
Z\mapsto H^{2p}_{|Z|}(X)\rightarrow H^{2p}(X).   
\end{displaymath}
Since $Z$ and $D$ does not meet, there is an isomorphism $H^{2p}_{|Z|}(X)\cong H^{2p}_{|Z|}(X,D)$ and the above map actually lands in relative cohomology
\begin{displaymath}
Z\mapsto H^{2p}_{|Z|}(X)\cong H^{2p}_{|Z|}(X,D_{red})\rightarrow H^{2p}(X,D_{red}).
\end{displaymath}
We call this the cycle class of $Z$ with modulus and use the symbol $cl_{p|D}(Z)$ to denote it.

Once we have this notion, the next task is to represent the relative cohomology groups $H^{\ast}(X,D_{red})$ using a suitable complex of differential forms. Since $D_{red}$ is an sncd, the hypercohomolgy associated with the sheaf of holomorphic forms $\Sigma_{D_{red}}\Omega_{X}^*$ vanishing on $D_{red}$ computes $H^{\ast}(X,D)$. Moreover, the inclusion $\Sigma_{D_{red}}\Omega^{\ast}_{X}\hookrightarrow \Sigma_{D_{red}}\caA^{\ast}_{X}$ into the sheaf of $C^{\infty}$-forms that vanish on $D_{red}$ is a quasi-isomorphism (see Remark \ref{rmk-non-reduced}). Since $\Sigma_{D_{red}}\caA^{\ast}_{X}$ is fine, the corresponding global section $\Sigma_{D_{red}}A^{\ast}(X)\coloneqq \Gamma(X, \Sigma_{D_{red}}\caA^{\ast}_{X})$ also computes $H^{\ast}(X,D)$. Hence, it is natural to use the complex $\Sigma_{D_{red}}A^{\ast}(X)$ to develop our theory. In particular, using Lemma \ref{lem:Surj*-2}, our first main result is the following:
\begin{theoalph}(Theorem \ref{thm:Surj})\label{th-A}
Let $Z\in \caZ^{p}(X|D)$. Then there exists a Green current $g_{Z}$ which satisfies the differential equation $\partial \overline{\partial}g_{Z}+\delta_{Z}=[\omega_{Z}]$, such that $\omega_{Z}\in \Sigma_{D_{red}}A^{p,p}(X)$, is real, satisfies $F^{\ast}_{\infty}\omega_{Z}=(-1)^{p}\omega_{Z}$, and its cohomology class $\{\omega_{Z}\}$ computes $cl_{p|D}(Z)$.    
\end{theoalph}
Two such Green currents always differ by a smooth form $u$ which satisfies $\partial\overline{\partial}u|_{E}=0$. Let's label the Green current of Theorem \ref{th-A} by $g_{Z|D}$. Then we define $\widehat{\caZ}^{p}(X|D)$ by pairs of the form $(Z, g_{Z|D})$. One can define a subgroup $\widehat{\caR}^{p}(X|D)$ of arithmetic rational equivalence with modulus generated by $\wh{{\rm div}(f)}\coloneqq ({\rm div}(f), -[\log|f|^{2}])$ for ${\rm div}(f)\in \caR^{p}(X|D)$, and pairs $(0, \partial u+\overline{\partial}v)$ for currents $u$ and $v$ of degrees $(p-2,p-1)$ and $(p-1,p-2)$, respectively. Hence, we have our main definition
\begin{defalph}
We define an arithmetic Chow group with modulus
\begin{displaymath}
\widehat{\CH}^{p}(X|D)\coloneqq \widehat{\caZ}^{p}(X|D)/\widehat{\caR}^{p}(X|D).
\end{displaymath}
\end{defalph}
Notice that, any current $T=\partial u+\overline{\partial}v$ automatically satisfy the condition that $\partial\overline{\partial}T=0$. Hence, we do not need to impose any vanishing condition on the complex of currents.

Much like the arithmetic Chow groups of Gillet and Soul\'e, the modulus version also satisfies the usual functoriality properties for smooth morphisms between smooth projective varieties (see Proposition \ref{functoriality}), and fits in an exact sequence (see Section \ref{sec:es}).

The Chow groups with modulus have a module structure over the classical Chow groups, which is defined by the usual product of algebraic cycles. It satisfies all the functorial properties, including the projection formula (\cite[Theorem 3.12]{KP17}). Correspondingly, at the arithmetic level, we get
\begin{theoalph}(Theorem \ref{thm:Prod-1})
There is an action 
\begin{displaymath}
  \cap\colon  \widehat{\CH}^p(X) \otimes \widehat{\CH}^q(X|D) \to \widehat{\CH}^{p+q}(X|D),
\end{displaymath}
which for a smooth morphism $\phi\colon X\to Y$ satisfies the projection formula.
\end{theoalph}

The isomorphism $\widehat{\Pic}(X)\cong\widehat{\CH}^{1}(X)$  between the group of isomorphism classes of hermitian line bundles and the arithmetic divisor class group, commuting with $\Pic(X)\cong\CH^{1}(X)$, is well-known.
On the modulus side, it is now natural to wonder about the existence of the relative version of the above isomorphism that commutes with 
$\Pic(X, D)\cong\CH^{1}(X|D)$. To address this question, 
we define the relative hermitian Picard group $\wh{\Pic}(X,D)$, consisting of triples $(L,\phi, ||\cdot||)$ such that the isomorphism class of $(L, \phi)$ belongs to $\Pic(X,D)$ and $||\cdot||$ is a hermitian metric on $L$ with the additional restriction that the smooth $(1,1)$-form representing its Chern class $c_{1}(L,||\cdot ||)$, vanish on $D_{red}$ (see Definition \ref{def:A-Pic-R-W*-1} for details). Then the following result is quite natural
\begin{theoalph}(Theorem \ref{thm:Chern-Iso})
There exists an isomorphism
\begin{displaymath}
\wh{\Pic}(X,D)\cong \widehat{\CH}^{1}(X|D) 
\end{displaymath}
which commutes with the isomorphism $\Pic(X,D)\cong\CH^{1}(X|D)$.
\end{theoalph}
\subsection{Layout}
Here we discuss the organization of the article. Sections \ref{sec:Ana-Pre}, \ref{sec:Rec-ACG}, and \ref{sec:CH-M} are overall preliminary in nature, where we gather known results that will be used for the later sections. We point out that Lemma \ref{lem:Surj*-2} in Section \ref{sec:Ana-Pre} plays a key role in the main part of the article. In Section \ref{sec*:ACH-M}, we introduce a definition of arithmetic Chow groups with modulus, establish the functoriality properties for them, and show that they fit in an exact sequence. In Section \ref{sec:Prodcut-ACGM}, we develop the module structure. Finally, in Section \ref{sec:codim-one}, we define a relative hermitian Picard group and prove that it is isomorphic to the arithmetic Chow group of divisors with modulus.

\subsection{Notations} \label{sec:Notations} 
For a quasi-projective scheme $X$ over a field $k$ and $i \geq 0$, we let $X^{(i)}$ (resp. $X_{(i)}$) denote the set of codimension (resp. dimension) $i$ points of $X$. If $X$ is an integral 
quasi-projective scheme over $k$, we let $X^N$ denote the normalization of $X$ in its function field $k(X)$. In some places, we use \textit{sncd} as the shorthand for a simple normal crossing divisor.

\section{Analytic component} \label{sec:Ana-Pre}
In this section, we gather the analytic results that will be used later in this article. We begin by recalling the well-known notions of differential forms and currents.
\subsection{Differential forms and currents}
Let $M$ be a complex manifold of complex dimension $d$. For $n \geq 0$, we let $A^n(M)$ denote the $\C$-vector space of 
complex-valued $C^{\infty}$-forms of degree $n$. 
Recall that we have a decomposition 
\begin{equation} \label{eqn:Com-form*-2}
    A^n(M) = \oplus_{p+q = n} A^{p,q}(M),
\end{equation}
where, locally on $M$, $A^{p,q}(M)$ is generated by the $n$-forms of the type $f({z}, {\ov{z}}) d z_I \wedge d \ov{z}_J$, with  $I, J \subset \{ 1, \dots, d\}$ such that $| I | = p$ and $ |J| = q$, and $f(z, \ov{z})$ is a complex valued $C^{\infty}$-function on $M$. 
Let $d \colon A^n(M) \to A^{n+1}(M)$ denote the differential map of $C^{\infty}$-forms. We then have maps $\partial \colon A^{p,q}(M) \to A^{p+1,q}(M)$ and 
$\ov{\partial} \colon A^{p,q}(M) \to A^{p,q+1}(M)$.  
such that
\begin{equation}\label{eqn:Com-form*-3}
    \partial \circ \partial = 0 = \ov{\partial} \circ \ov{\partial} \textnormal{ and } d= \partial + \ov{\partial}. 
\end{equation}
Moreover, 
${\rm Ker}(\ov{\partial}^{0,0}) = \sO(M) \subset A^0(M)$,
where $\sO(M)$ denotes the complex of holomorphic functions on $M$. More generally, an element $\phi\in A^{p,0}(M)$ is holomorphic if $\overline{\partial}\phi=0$. This is the case if locally around a point $z\in M$ we can write
\begin{displaymath}
\phi(z)=\Sigma_{|I|=p}\phi_{I}(z)dz_{I},~\phi_{I}\in \sO(M).
\end{displaymath}

For $n\geq 0$, we let $A^n_c(M) \subset A^n(M)$ denote the subgroup of the compactly supported complex valued $C^{\infty}$-forms. We let  $A^{p,q}_c(M) = A^{p+q}_c(M) \cap A^{p,q}(M)$. 
 Note that an element $\omega \in A^{p,q}_c(M)$ is locally of the type $f({z}, {\ov{z}}) d z_I \wedge d \ov{z}_J$, where $f({z}, {\ov{z}})$ is supported on a compact subset. 
Conversely, if a form $\omega' \in A^{p, q}(M)$ is (locally) of the above type, then $\omega' \in A^{p,q}_c(M)$. 

We say a linear functional $T \colon A^n_c(M) \to \C$ is continuous in
the sense of Schwartz if for a sequence $\{\omega_n\}$ of compactly supported complex valued $C^{\infty}$-forms that are supported on a fixed compact subset and $\omega_n \to 0$, we have 
$T(\omega_n) \mapsto 0$, where $\omega_n \to 0$ means that in the local expression, the $C^{\infty}$-functions 
$f_n(z, \ov{z})$ along with finitely many of their derivatives tends to $0$. 
The set of such linear functional forms a group that we denote by 
$D^{2d-n}(M) :=  A^n_c(M)^*$. 
Similarly, we define $D^{p, q}(M) = A^{d-p, d-q}_c(M)^*$. Then
 $D^{n}(M) = \oplus_{p+q =n} D^{p,q}(M)$. 
An element of $D^n(M)$ is called a current on $M$.  

 Since 
 $\partial\colon A^{d-p-1, d-q}_c(M) \to A^{d-p, d-q}_c(M)$ respects continuity in
the sense of Schwartz, we have an induced map 
 $\partial\colon D^{p, q}(M) \to D^{p+1, q}(M)$
 so that $(\partial T)(\omega) = (-1)^{p+q+1} T( \partial{\omega})$. Similarly, we can define 
 $\ov{\partial} \colon D^{p, q}(M) \to D^{p, q+1}(M)$ and 
 $d=\partial+\overline{\partial}$ such that the current variant of \eqref{eqn:Com-form*-3} holds. We use the notation $\widetilde{A}^{\ast}(X)$ (resp. $\widetilde{D}^{\ast}(X)$) to denote $A^{\ast}(X)/\im(d)$ (resp. $D^{\ast}(X)/\im(d)$), and $\widetilde{A}^{p,q}(X)$ (resp. $\widetilde{D}^{p,q}(X)$) to denote $A^{p,q}(X)/\im(\partial)+\im(\overline{\partial})$ (resp. $D^{p,q}(X)/\im(\partial)+\im(\overline{\partial})$).

\begin{lem} \label{lem:form--current-inc}
With notations as above, there exists a canonical injective map 
\begin{equation} \label{eqn:form-current-inc}
    [~]\colon  A^{p. q}(M) \inj D^{p, q}(M)
\end{equation}
so that for $\omega_1 \in A^{p,q}(M)$ and $\omega_2 \in  A_c^{d-p, d-q}(M)$, we have 
\begin{equation} \label{eqn:form-current-inc*-2}
[\omega_1](\omega_2) = \int_M  \omega_1 \wedge \omega_2. 
\end{equation}
Moreover, it respects maps $\partial$, $\ov{\partial}$ and hence $d$. 
\end{lem}

Recall that $A^n(M) = A^n_{{\rm real}}(M) \otimes_{\R} \C$, where 
$A^n_{{\rm real}}(M)$ consists of the real $C^{\infty}$-forms on $M$ (considered as real manifold). We let 
$\tau \colon A^n(M) \to  A^n(M)$ denote the involution induced by the conjugation on $\C$. Note that the map $\tau$ takes compactly supported forms to themselves and hence induces an involution $\tau \colon A^{n}_c(M) \to  A^{n}_c(M)$. We define $\tau \colon D^{n}(M) \to  D^{n}(M)$ so that for $T \in D^n(M)$ and $\omega \in A^n_c(M)$, we have
\begin{equation} \label{eqn:conj-currents-forms}
\tau(T) (\omega) = \ov{T (\tau (\omega))} \in \C,
\end{equation}
where $\ov{(~)}$ denotes the complex conjugation. Note that since $\tau \colon A^{n}_c(M) \to  A^{n}_c(M)$ is conjugate linear, the map $\tau(T)$ is a linear map on $A^n_c(M)$ that is continuous in the sense of Schwartz, i.e., 
$\tau(T) \in D^n(M)$.  
It follows from the definition that for $p, q\geq 0$, we have 
\begin{equation} \label{eqn:real-forms/curr*-1}
\tau(A^{p,q}(M)) = A^{q, p}(M) \textnormal{ and }
\tau(D^{p,q}(M)) = D^{q, p}(M). 
\end{equation}
In particular, for each $p \geq 0$, the map $\tau$ induces involutions on $A^{p, p}(M)$ and $D^{p,p}(M)$.  

\begin{df} \label{defn:real-(p,p)-forms}
For $p\geq 0$,  the involution $\tau \colon A^{p, p}(M) \to A^{p,p}(M)$ (resp. $\tau \colon D^{p, p}(M) \to D^{p,p}(M)$)  is called the complex conjugation on $A^{p,p}(M)$ (resp. $D^{p, p}(M)$) and its fixed points are called as real $(p,p)$-forms (resp. real $(p,p)$-currents).
\end{df}

It is easy to see that the map \eqref{eqn:form-current-inc} takes real $(p,p)$-forms to real $(p,p)$-currents. Before proceeding further, we recall examples of real currents.

\begin{ex} \label{ex:curr-man}
\begin{enumerate}
\item 
Let $Z\subset M$ be an analytic subspace of $M$ of codimension $j$
and let $Z_{{\rm ns}}$ denote the dense open subset of the smooth points of $Z$. Then $Z_{{\rm ns}}$ is a complex manifold of complex dimension $d-j$. 
We define $\delta_Z \in D^{2j}(M)$ so that for $\omega \in A^{2d-2j}_c(M)$, we have 
\begin{equation} \label{eqn:Ex-Ana-sub}
    \delta_Z (\omega) := \int_{Z_{{\rm ns}}} \phi^*(\omega), 
\end{equation}
where $\phi\colon Z_{{\rm ns}} \inj Z \inj M$ is the inclusion map of complex manifolds. We can easily see that $\delta_Z \in D^{j, j}(M)$ and it is real $(j,j)$-current. 
Moreover, if $\pi\colon \wt{Z} \to Z$ is a resolution of singularity of $Z$, then 
\begin{equation}\label{eqn:Ex-Ana-sub*-1}
    \delta_Z (\omega) = \int_{\wt{Z}} \wt{\pi}^* (\omega),
\end{equation}
where 
$\omega \in A^{2(d-j)}_c(M)$ and $\wt{\pi} \colon \wt{Z} \to M$ is the induced map between complex manifolds. We can extend it linearly to analytic cycles.

\item  Let 
$L_{{\rm loc}}^1(M)$ denote the locally $L^1$-functions on $M$ and define $L^1$-forms of degree $(p,q)$ as 
$L^1_{{\rm loc}}(A^{p,q}(M)) = L^1_{{\rm loc}}(M)_{\C} \otimes_{A^0(M)} A^{p,q}(M)$. 
We then have a well-defined map 
$[~]\colon L^1_{{\rm loc}}(A^{p,q}(M)) \to D^{p, q}(M)$ so that for $\alpha \in L^1_{{\rm loc}}(A^{p,q}(M))$ and $\omega' \in A_c^{d-p, d-q}(M)$, we have 
\begin{equation} \label{eqn:Ex-L-1-curr}
[\alpha] (\omega') = \int_M \alpha \wedge \omega'.     
\end{equation}
In a more specific case, let $f\colon M \to \C$ be a meromorphic function on $M$, then the function 
$-\log |f|^2$ defines a locally $L^1$-function on $M$. We associate the current $[- \log |f|^2] \in D^0(M)$ to the meromorphic function $f$ which satisfies 
\begin{equation} \label{eqn:Div-An*-0}
(\iota/2\pi) \partial \ov{\partial} [- \log |f|^2] =- \delta_{\Div(f)}.
\end{equation}
\end{enumerate}
\end{ex}

\subsection{Extension lemma and relative forms}
We now prove an extension lemma for $C^{\infty}$-differential forms. 

\begin{lem}[Extension Lemma] \label{lem:Ext-Lem-I}
Let $N$ be a closed complex submanifold of $M$. Then the pull-back map 
$A^*(M) \to A^*(N)$ is surjective. 
\end{lem}
\begin{proof}
    By the definition of $A^*(-)$, it suffices to prove the lemma for the $C^{\infty}$-functions, i.e., for $* =0$. But this follows from \cite[Lemma~5.34]{Lee03}. For the sake of completion, we write a proof here. 
    
    Let $\omega \in A^n(N)$. Let $\{U_{\alpha}\}_{\alpha}$ be a chart of $M$ such that $\{U_{\alpha} \cap N\}_{\alpha}$ is a chart of $N$. Let $z_{\alpha, 1}, \dots, z_{\alpha, d}$ be complex co-ordinates of $U_{\alpha}$ and let $z_{\alpha,1}, \dots, z_{\alpha, e}$ be the co-ordinates of $U_{\alpha}\cap N$. For all $I, J \subset \{1, \dots, e\}$ such that $|I|+|J|=n$, there exist $C^{\infty}$-functions, $f_{\alpha, I, J}$ on $U_{\alpha}$ such that 
    \[
    \omega|_{U_{\alpha}\cap N} = \sum_{I, J} f_{\alpha, I, J} dx_{\alpha, I} \wedge d y_{\alpha, J} \in A^n(N\cap U_{\alpha}). 
    \]
    As after possibly shrinking $U_{\alpha}$, we can assume that each $U_{\alpha}$ is isomorphic to a disc $\Delta(l_{\alpha})$ of radius $l_{\alpha}$ and under the restriction of this isomorphism, 
    $U_{\alpha} \cap N$ is isomorphic to  
    $\Delta'(l_{\alpha}) = \{(z_1, \dots, z_e, 0, \dots, 0) \in \Delta(l_{\alpha})\}$. Pulling back the functions $f_{\alpha, I, J}$ under the projection $\Delta(l_{\alpha}) \to  \Delta'(l_{\alpha})$, we get $C^{\infty}$-functions $h_{\alpha, I, J}$ defined on $U_{\alpha}$ such that $(h_{\alpha, I, J})|_{U_{\alpha}\cap N} = f_{\alpha, I, J}$. 
    We let 
    \[
    \omega_{\alpha} = \sum_{I, J} h_{\alpha, I, J} dx_{\alpha, I} \wedge d y_{\alpha, J} \in A^n(U_{\alpha}).
    \]
    By definition, it follows that $\omega_{\alpha}|_{U_{\alpha}\cap N} = \omega|_{U_{\alpha}\cap N}$. We now let $\{\phi_{\alpha}\}_{\alpha}$ be a partition of unity subordinate to $\{U_{\alpha}\}_{\alpha}$ and let 
    \[
    \eta = \sum_{\alpha} \phi_{\alpha} \omega_{\alpha} \in A^{n}(M). 
    \]
    Note that $\phi_{\alpha} \omega_{\alpha}$ is defined on $M$ so that $\phi_{\alpha} \omega_{\alpha}$ is equal to the well-defined product $(\phi_{\alpha}|_{U_{\alpha}}) \omega_{\alpha}$ on $U_{\alpha}$ and it is zero on $M \setminus {\rm supp}(\phi_{\alpha})$ (because $U_{\alpha} \cup (M \setminus {\rm supp}(\phi_{\alpha})) = M$). It is therefore an element of $A^n(M)$. 
    We are now done because we have
    \begin{eqnarray*}
    \eta|_{N} &=& \sum_{\alpha} (\phi_{\alpha}|_{N}) ( \omega_{\alpha}|_{N\cap U_{\alpha}})\\
    & = &\sum_{\alpha} (\phi_{\alpha}|_{N}) ( \omega|_{N\cap U_{\alpha}}) \\
    & =& \omega,
    \end{eqnarray*}
   where the last equality holds because $\sum_{\alpha} \phi_{\alpha} =1$. 
\end{proof}

The kernel of the map in Lemma \ref{lem:Ext-Lem-I} is given by $C^{\infty}$-forms on $M$ that vanish on $N$. 
Recall that an analytic space $N = \cup_{i=1}^m N_i$ is a normal crossing manifold if each irreducible component $N_i$ is a smooth complex manifold and the irreducible components of their intersections $N_I = \cap_{i\in I}N_i$ are also smooth complex manifold of the correct dimension, where $I\subset \{1,\dots, m\}$. 
%In particular, a point $x\in N$ is contained in a neighborhood $U_{x}\cong \{(z_{0},\dots, z_{d+1})\in \Delta^{d}(1)\colon z_{0}\cdots z_{k}=0\}$.
In the remainder of this section, $I$ will denote a subset of $\{1,\dots, m\}$, and we denote by $|I|$ the cardinality of $I$. We denote by $K^{p,*}_N$ the total complex associated to the double complex
\[
\oplus_{|I|=1} A^{p,\ast}(N_I) \to \oplus_{|I|=2} A^{p,\ast}(N_{I}) \to \cdots \to A^{p,\ast}(\cap^{m}_{i=1}N_{i})\to 0.
\]
Below, we can generalize Lemma \ref{lem:Ext-Lem-I} to the case where $N$ is a simple normal crossing divisor of $M$.

\begin{lem} \label{lem:Surj*-2}
 Let $M$ be a complex manifold and let $N$ be a simple normal crossing divisor of $M$. Let $\Sigma_N A^{p,*}_M := \textnormal{Ker}(A^{p,*}(M) \to \oplus_i A^{p,*}(N_i))$. Then for each $p,q$, the following sequence is exact.
\begin{equation} \label{eqn:Surj*-2.2}
0 \to  \Sigma_N A^{p,q}_M \to A^{p,q}(M)  \to \oplus_{|I|=1} A^{p,q}(N_I) \to \oplus_{|I|=2} A^{p,q}(N_{I}) \to \cdots \to A^{p,q}(\cap^{m}_{i=1}N_{i})\to 0.
\end{equation}
 In particular, the natural map
 \begin{equation} \label{eqn:Surj*-2.1}
     \Sigma_N A^{p,*}_M  \to K^{p,*}_{M,N}
 \end{equation}
 is a quasi-isomorphism, where $K^{p,*}_{M,N} := s(A^{p,*}(M) \to K^{p,*}_N)$ is the associated total complex. 
\end{lem}
\begin{proof}
 We first note that by a standard spectral sequence argument, the exactness of the sequence \eqref{eqn:Surj*-2.2} for all $q$ (and a fix $p$) implies the quasi-isomorphism \eqref{eqn:Surj*-2.1}. We are therefore reduced to showing the exactness of \eqref{eqn:Surj*-2.2}. Note that for a morphism $N' \to M'$ of smooth complex manifolds,  the pull-back map $A^m(M') \to A^m(N')$ takes the direct summnad $A^{p,q}(M')$ to the direct summand $ A^{p,q}(N')$ (for $p,q$ such that $p+q=m$). It therefore suffices to show the following sequence is exact. 
  \begin{equation} \label{eqn:Surj*-2.2.5}
0 \to  \Sigma_N A^{m}_M \to A^{m}(M)  \to \oplus_{i} A^m(N_i) \to \oplus_{\{i, j\}} A^m(N_{ij}) \to \oplus_{\{i, j, k\}} A^{m}(N_{ijk}) \cdots
  \end{equation}

    We first prove the exactness of \eqref{eqn:Surj*-2.2.5} in the case when $N$ is also a smooth manifold. Note that it suffices to show that the pull-back map $A^m(M) \to A^m(N)$ is surjective. But this follows from Lemma~\ref{lem:Ext-Lem-I} (also see \cite[Lemma~5.34]{Lee03}). Indeed, we proved in the referred lemma that the pull-back map $A^m_{{\rm real}}(M) \to A^m_{{\rm real}}(N)$ is surjective. But then the claim follows because 
    $A^m(M) = A^m_{{\rm real}}(M) \otimes_{\R} \C$. 

    We now prove the exactness of \eqref{eqn:Surj*-2.2.5} in the general case when $N= \cup_{i=1}^n N_i$ is a normal crossing manifold. 
    By a use of partition of unity, one can see that the above problem is local on $M$, i.e., it suffices to prove the lemma when $M$ is an open ball in $\C^d$ (see the end of the proof of Lemma~\ref{lem:Ext-Lem-I}).  We now work under this hypothesis and prove the lemma by induction on the dimension $d \geq 1$ of the manifold $M$ and $n \geq 1$, the number of components of $N$. The base cases when $n=1$ or $d =1 $ follow from the above paragraph. We assume that $d, n \geq 2$ and show that the sequence \eqref{eqn:Surj*-2.2} is exact. We first prove that the sequence 
     \begin{equation} \label{eqn:Surj*-2.3}
     A^m(M)  \to \oplus_{i} A^m(N_i) \to \oplus_{i, j} A^m(N_{ij}) 
     \end{equation}
    is exact and then use it to show the exactness of \eqref{eqn:Surj*-2.2.5}. 

    Recall that we are proving the lemma in the local case. Let $(z_1, \dots, z_d)$ be a co-ordinate system of $M$ (which is currently an open ball in $\C^d$) such that $N = V(z_1z_2 \cdots z_n) \subset M$.
    We let $N_i = V(z_i)$ and let $\un{\omega}= (\omega_1, \cdots, \omega_n) \in \oplus_{i} A^m(N_i)$ such that $(\omega_i)|_{N_{ij}} = (\omega_j)|_{N_{ij}}$ for all $i, j \in \{1, \dots, n\}$. By the induction hypothesis, we can find a smooth from $\omega' \in A^m(M)$ such that for all $2\leq j \leq n$, we have
     \begin{equation} \label{eqn:Surj*-2.4}
     (\omega')|_{N_j} = \omega_j \in A^m(N_j). 
     \end{equation}
     We let $\omega_1' = (\omega')|_{N_1} - \omega_1 \in A^m(N_1)$. Note that $(\omega_1')|_{N_{1j}} = 0$, for all $2\leq j \leq n$ because of the hypothesis on $\un{\omega}$. We now use the projection map ${\rm pr}_1 \colon M\to N_1$ which maps $(z_1, \dots, z_d) \mapsto (z_2, \dots, z_d)$ and take 
     $\omega'' = {\rm pr}_1^*(\omega_1') \in A^m(M) $.  We then have 
     \begin{equation}\label{eqn:Surj*-2.5}
         (\omega'')_{N_1} = \omega'_1 =  (\omega')|_{N_1} - \omega_1 \in A^m(N_1) \textnormal{ and } (\omega'')_{N_j} = 0 \in A^m(N_j)~ \forall j \geq 2.  
     \end{equation}
     Note that the last equality holds because 
     $(\omega'')|_{N_j} =({\rm pr}_1^*(\omega_1'))|_{N_j} = ({\rm pr}^1_{1j})^*(\omega_1')|_{N_{1j}} =0$, where ${\rm pr}^1_{1j} \colon N_j \to N_{1j}$ is the induced projection map. We are now done by \eqref{eqn:Surj*-2.4} and \eqref{eqn:Surj*-2.5}. Indeed, if we let 
     $\wt{\omega} = \omega' - \omega'' \in A^m(M)$, then 
      \begin{equation*} \label{eqn:Surj*-2.6}
         (\wt{\omega})|_{N_1} =  (\omega')_{N_1} - (\omega'')_{N_1} = (\omega')_{N_1} - ((\omega')_{N_1} - \omega_1) = \omega_1
    \end{equation*}
    and for all $2 \leq j \leq n$
    \begin{equation*} \label{eqn:Surj*-2.7}
         (\wt{\omega})|_{N_j} = (\omega')_{N_j} - (\omega'')_{N_j} =  (\omega')_{N_j} - 0 = \omega_j.
    \end{equation*}
    This completes the proof of the exactness of the sequence \eqref{eqn:Surj*-2.3}. Finally, we prove below the exactness of the sequence  \eqref{eqn:Surj*-2.2.5}. 

    We start by fixing certain notations. For $0 \leq r\leq n$, we let $\sI_r$ denote the set of subsets of $\un{n} = \{1, \dots, n\}$ of cardinality $r$. We then need to show the exactness of the following sequence. 
    \begin{equation} \label{eqn:Surj*-2.8}
    \oplus_{I \in \sI_{r-1}} A^m(N_I) \xrightarrow{\partial_r} \oplus_{I \in \sI_{r}} A^m(N_I) \xrightarrow{\partial_{r+1}} \oplus_{I \in \sI_{r+1}} A^m(N_I),
    \end{equation}
    where as before $N_J = \cap_{j\in J} N_j$ and $N_{\emptyset} = M$. Note that we have proven the exactness of \eqref{eqn:Surj*-2.8} when $r=1$ and need to show its exactness for $r\geq 2$. Let $\un{\omega} = (\omega_I)_{I\in \sI_r} \in \oplus_{I \in \sI_{r}} A^m(N_I)$ such that it maps to $0 \in \oplus_{I \in \sI_{r+1}} A^m(N_I)$.
    %{\sl The main idea to proof the exactness of \eqref{eqn:Surj*-2.8} is as follows. We first work away from $N_1$ using induction on $r$. This works because if we just take coordinates that do not involve $1$, the element still maps to zero under the second map in \eqref{eqn:Surj*-2.8}, see \eqref{eqn:Surj*-2.8.2} below. Using induction on $r$, we then take, up to the image of the first map in \eqref{eqn:Surj*-2.8}, all coordinates of $\un{\omega}$ that do not involve $1$ to be zero. But then we can work inside $N_1$ because the second map in \eqref{eqn:Surj*-2.8} for $M$ and for $N_1$ are the same on the elements whose coordinates that do not involve $1$ are zero. We are then done by an induction on dimension of $M$.} 
    Let $\sI_{r\setminus 1} = \{I \in \sI_r | 1 \notin I\}$. Then by the induction hypothesis on $r$, the sequence
    \begin{equation} \label{eqn:Surj*-2.8.1}
    \oplus_{I \in \sI_{r-1\setminus 1}} A^m(N_I) \xrightarrow{\partial_r'} \oplus_{I \in \sI_{r\setminus 1}} A^m(N_I) \xrightarrow{\partial_{r+1}'} \oplus_{I \in \sI_{r+1\setminus 1}} A^m(N_I)
    \end{equation}
    is exact. Note that 
    \begin{equation} \label{eqn:Surj*-2.8.2}
    \partial_{r+1}' ((\omega_J)_{J\in \sI_{r\setminus 1}}) = (\partial_{r+1}(\underline{\omega}))_{I \in \sI_{r+1\setminus 1}} = 0. 
    \end{equation}
    Therefore 
    there exists 
    $\un{\omega}' = 
    (\omega'_I)_{I\in \sI_{r-1\setminus 1}} \in \oplus_{I \in \sI_{r-1\setminus 1}} A^m(N_I)$ such that for $J\in \sI_{r\setminus 1}$, we have
    \begin{equation} \label{eqn:Surj*-2.8.3}
    (\partial'_{r} (\un{\omega}'))_J = \sum_{i\in J} (-1)^i (\omega'_{J\setminus \{i\} })|_{N_J} = \omega_J \in A^m(N_J).
    \end{equation}
    Let $\un{\omega_1}=\un{\omega} - \partial_r((\un{\omega}', 0)) \in  \oplus_{J \in \sI_{ r}} A^m(N_J)$, where we put $0$ on the components corresponding to $J \in \sI_{1, r} := \sI_{r-1} \setminus\sI_{r-1\setminus 1}$ (i.e., when $1 \in J$). Note that the $J$-th component $(\un{\omega_1})_J$ of $\un{\omega_1}$ is zero if $J \in \sI_{r\setminus 1}$. Let us write $\un{\omega_1} = (0, \underline{\omega_2})$, where $\underline{\omega_2} \in \oplus_{I \in \sI_{1, r}} A^m(N_I)$ and the components corresponding to $J \in \sI_{r\setminus 1}$ are zero.

    We now work with the smaller dimension submanifold $N_1$ and the normal crossing divisor $\cup_{1< j \leq n} (N_1 \cap N_j)$ on $N_1$. By the induction hypothesis on the dimension of the manifold, the sequence
    \begin{equation} \label{eqn:Surj*-2.8.4}
    \oplus_{I \in \sI_{1, r-1}} A^m(N_I) \xrightarrow{\partial_r''} \oplus_{I \in \sI_{1, r}} A^m(N_I) \xrightarrow{\partial_{r+1}''} \oplus_{I \in \sI_{1, r+1}} A^m(N_I)
    \end{equation}
    is exact. Note that
    \begin{eqnarray*}
        \partial_{r+1}'' (\underline{\omega_2}) = \partial_{r+1}((0, \underline{\omega_2})) =  \partial_{r+1}(\underline{\omega_1}) =  \partial_{r+1}(\un{\omega} - \partial_r((\un{\omega}', 0)))= \partial_{r+1}(\un{\omega}) = 0. 
    \end{eqnarray*}
    In particular, there exists $\un{\omega}'' = 
    (\omega''_I)_{I\in \sI_{1, r-1}} \in \oplus_{I \in \sI_{1, r-1}} A^m(N_I)$ such that for $J\in \sI_{1, r}$, we have
    \begin{equation} \label{eqn:Surj*-2.8.5}
    (\partial''_{r} (\un{\omega}''))_J = \sum_{1 \neq i\in J} (-1)^i (\omega''_{J\setminus \{i\} })|_{N_J} = (\omega_2)_J \in A^m(N_J).
    \end{equation}
    We let $\underline{\widetilde{\omega}} = (\un{\omega}', \un{\omega}'') = 
    (({\omega}'_I)_I, ({\omega}''_J)_J) \in
    \oplus_{I \in \sI_{r-1\setminus 1}} A^m(N_I)\oplus \oplus_{J \in \sI_{1, r-1}} A^m(N_J) = \oplus_{I \in \sI_{r-1}} A^m(N_I)$. We then have 
    \begin{eqnarray*}
        \partial_r(\underline{\widetilde{\omega}}) & =^{} & 
         \partial_r(\underline{\omega}', 0) + \partial_r(0, \underline{\omega}'') \\
         & =^1 & \partial_r(\underline{\omega}', 0) + (0, \partial_r''( \underline{\omega}'')) = \omega \in \oplus_{I \in \sI_{r}} A^m(N_I),
    \end{eqnarray*}
    where the equality one holds because 
    $(\partial_r(0, \underline{\omega}''))_I  = 0$ if $1\notin I$ and by the calculation \eqref{eqn:Surj*-2.8.5},  
    $(\partial_r(0, \underline{\omega}''))_I  = \sum_{1 \neq i\in I} (-1)^i (\omega''_{I\setminus \{i\} })|_{N_I} =(\omega_2)_J$ for $1\in I \in \sI_r$. This completes the proof. 
\end{proof}

\begin{rmk} \label{rmk-non-reduced}
We summarize and extract the key essence of Lemma \ref{lem:Surj*-2} in this remark. Let for an analytic subvariety $N$, $\Sigma_{N}\Omega^{\ast}_{M}$ and $\Sigma_{N}\caA^{\ast}_{M}$ denotes the sheaves of holomorphic and $C^{\infty}$-forms on $M$ that vanishes on $N$ respectively. When $N$ is a sncd, it is well known that the inclusion $\Sigma_{N}\Omega^{\ast}_{M}\hookrightarrow \Sigma_{N}\caA^{\ast}_{M}$ is a quasi-isomorphism. This follows from the fact that the inclusion $\Omega^{\ast}_{N_{I}}\hookrightarrow \caA^{\ast}_{N_{I}}$ is a quasi-isomorphism for each $I$ (since $N_{I}$ is smooth for each $I$), and one can obtain a commutative diagram involving a sheaf theoretic version of \eqref{eqn:Surj*-2.2}. Since $\Sigma_{N}\caA^{\ast}_{M}$ is a fine sheaf, the upshot is that the cohomology of the global sections $\Sigma_{N}A^{\ast}_{M}$ computes the relative cohomology groups $H^{\ast}(M,N)$. In the world of algebraic cycles, the cycle class of an algebraic cycle is always given by a real $C^{\infty}$-form. This forces us to work with an sncd $N$, until we find a variant of $\Sigma_{N}\caA^{\ast}_{M}$ for a general (even non-reduced) analytic subvariety $N$ (see \S8 of \cite{Binda-Saito} for such a sheaf at the holomorphic level).
\end{rmk}

\section{Arithmetic Chow groups} \label{sec:Rec-ACG}

In this section, we recall Arithmetic Chow groups as defined in \cite{GS90}. We first give some details about the conjugate linear involution induced by the complex conjugation on $\C$. Even though these details are well-known, it will be helpful for the reader to see them in action through concrete examples and computations.
\subsection{Conjugate linear involution $F_{\infty}$ and arithmetic rings}

In this subsection, we briefly recall the definition of arithmetic rings. We discuss various examples of such rings. 

\begin{df}(\cite[Def~3.1.1]{GS90})
By an arithmetic ring $A= (A, \Sigma, F_{\infty})$, we mean a subring $A$ of complex numbers with a set $\Sigma$ of (ring) monomorphisms  $\sigma \colon A\inj \C$ and a conjugate linear involution 
$F_{\infty} \colon \C^{\Sigma} \to \C^{\Sigma}$ such that its restriction on $A$ is identity. 
\end{df}
\begin{ex}   \label{ex:F-infty}
Let $k$ be a number field such that $[k:\mathbb{Q}]=r+2s$, and 
\begin{displaymath}
\Sigma\coloneqq \{\sigma_{1},\dots,\sigma_{r}, \tau_{1},\overline{\tau_{1}},\dots, \tau_{s},\overline{\tau_{s}}\},
\end{displaymath}
where $\sigma_{i}$-s are the real and $\tau_{j}$-s are the complex embeddings of $k$. Then $F_{\infty}$ acts on $ \C^{\Sigma} = \prod_i \C^{\sigma_i} \times \prod_j (\C \times \C)^{\tau_j}$ by complex conjugation on the copies of $\C^{\sigma_i}$ and by acting on $(\C \times \C)^{\tau_j}$ so that 
$(z_1, z_2)^{\tau_j} \mapsto (\ov{z_2}, \ov{z_1})^{\tau_j}$ for each $j$. Following the same idea, for $k=\R$ the action is given by complex conjugation, while for $k=\C$, it is given by $(z_{1},z_{2})\mapsto (\overline{z_{2}},\overline{z_{1}})$.
\end{ex}
In general, we have the following description about the $F_{\infty}$-map. 
\begin{lem} \label{lem:F-infty-C-Sigma}
Given a conjugate linear map ring homomorphism $F_{\infty} \colon \C^{\Sigma} \to \C^{\Sigma}$ that is involution, there exists a bijection $\Theta \colon \Sigma \to \Sigma$ such that $F_{\infty} (z_{\sigma}) = (\ov{z}_{\Theta(\sigma)})$. 
\end{lem}

We now move to algebraic geometry and consider the smooth varieties over arithmetic fields. 
\begin{df}(\cite[Def 3.2.1]{GS90}) \label{def:X-infty}
Let $k = (k, \Sigma, F_{\infty})$ be an arithmetic field and let $f \colon X \to \Spec(k)$ be a smooth quasi-projective $k$-scheme of dimension $d \geq 0$. For $\sigma \in \Sigma$, we let 
$X_{\sigma} =  X \times_{\sigma} \Spec(\C)$ and let $X_{\Sigma} = \coprod_{\sigma \in \Sigma} X_{\sigma}$. Since $X$ is smooth over $k$, it follows that $X_{\Sigma}$ is a product of smooth quasi-projective varieties over $\C$. We let $X_{\infty}$ denote the smooth complex manifold 
$X_{\Sigma}(\C) =  \coprod_{\sigma \in \Sigma} X_{\sigma}(\C) = X(\C^{\Sigma})$.  We let $F_{\infty} \colon X_{\infty} \to X_{\infty}$ denote the induced diffeomorphism. 
\end{df}
Let $F^*_{\infty} \colon A^n(X_{\infty}) = A^n_{{\rm real}}(X_{\infty}) \otimes_{\R} \C \to A^n(X_{\infty})$ denote the $\C$-linear involution defined by  $F_{\infty} \colon X_{\infty} \to X_{\infty}$. Moreover, we also have involution $F_{\infty}^* \colon D^{n}(X_{\infty}) \to  D^{n}(X_{\infty})$ such that for $\omega \in A_c^{2d-n}(X_{\infty})$ and $T \in D^{n}(X_{\infty})$, we have
$F_{\infty}^*(T) (\omega) = T((F_{\infty})_* (\omega))$, where $(F_{\infty})_*(\omega)$ is the push-forward of the form $\omega$ along the diffeomorphism $F_{\infty}$. 
 
\begin{lem} \label{lem:F-infty-forms}
For $p, q$, we have 
%\begin{equation} \label{eqn:F-inf-form*-1}
$F_{\infty}^*(A^{p,q}(X_{\infty})) = A^{q, p}(X_{\infty})$ 
and $F_{\infty}^*(D^{p,q}(X_{\infty})) = D^{q, p}(X_{\infty})$ . 
%\end{equation}
In particular, for all $p$, $F^*_{\infty}$ induces involutions on $A^{p, p}(X_{\infty})$ and $D^{p, p}(X_{\infty})$. Moreover, 
$F^*_{\infty}({\iota}/{2\pi}   \partial \ov{\partial})= (-1) ({\iota}/{2\pi}   \partial \ov{\partial})$. 
\end{lem}
\begin{proof}
    We note that by Lemma \ref{lem:F-infty-C-Sigma}, there exist a bijection $\Theta \colon \Sigma \to \Sigma$ and trivializations
$V^{\sigma} \subset \R^{2d}$ of $X_{\sigma}$, for $\sigma \in \Sigma$, with co-ordinates $(x^{\sigma}_1, y^{\sigma}_1, \cdots, x^{\sigma}_d, y^{\sigma}_d)$ such that 
$F_{\infty}(V^{\sigma}) \subset V^{\Theta(\sigma)}$ and 
\begin{equation} \label{eqn:F-inf-form*-1}
F_{\infty}^* (x^{\Theta(\sigma)}_1, y^{\Theta(\sigma)}_1, \dots, x^{\Theta(\sigma)}_d, y^{\Theta(\sigma)}_d) = (x^{\sigma}_1, - y^{\sigma}_1, \dots, x^{\sigma}_d, - y^{\sigma}_d).
\end{equation}
Since the claims in the lemma are local on $X$, the lemma follows from \eqref{eqn:F-inf-form*-1} and easy calculations. 
 
\end{proof}

\begin{df} \label{def:real-F-infty-forms-currents}
Let $k = (k, \Sigma, F_{\infty})$ be an arithmetic field and let $f \colon X \to \Spec(k)$ be a smooth quasi-projective $k$-scheme of dimension $d \geq 0$. For all $p\geq 0$, we let 
\begin{equation} \label{eqn:R-F-int-Forms*-1}
A^{p,p}(X_{\R}) = \{ \omega \in A^{p,p}(X_{\infty}) | \textnormal{ $\omega$ is real (Definition \ref{defn:real-(p,p)-forms}) and}~F_{\infty}^{*}(\omega) = (-1)^p \omega\},
\end{equation}
\begin{equation} \label{eqn:R-F-int-Forms*-2}
A^{p,p}_c(X_{\R}) = A^{p,p}(X_{\R}) \cap A^{2p}_c(X_{\infty})
\end{equation}
and 
\begin{equation}\label{def:real-F-infty-currents}
D^{p,p}(X_{\R}) = \{ T \in D^{p,p}(X_{\infty})~|~T~\textnormal{is real}, F_{\infty}^{*}(T) = (-1)^p T\}. 
\end{equation}
\end{df}

We now see the action of $F_{\infty}$ on some prototypical differential forms and currents. These examples make the above definitions natural.

\begin{ex} \label{ex:real-currents}
    \begin{enumerate}
\item
For $X=\P^1_{\Q}$, the projective line over $\Q$, the differential form $\omega=\frac{dz\wedge d\bar z}{(1+|z|^2)^2}$ on $X_{\infty}$ is an element in $A^{1,1}(X_{\R})$.
        \item  For an irreducible closed subset $Z$ of codimension $p$, we let 
$\delta_Z \in D^{p,p}(X_{\infty})$ denote the current associated to the analytic subspace $Z_{\infty} \subset X_{\infty}$ (see Example \ref{ex:curr-man}(ii)). Then $\delta_Z$ is a real $(p,p)$-current.  

Using \cite[Prop.XII, Pg.162]{G-H-V}, one can prove that 
\begin{equation} \label{ex:real-current-delta-Z}
     F_{\infty}^* \delta_Z = \epsilon_1  ~ \delta_Z \circ F_{\infty}^* ( F_{\infty})_* 
\end{equation}
where $\epsilon_1 = 1$ (resp. $-1$) if ${}^N F_{\infty} \colon N \to N$ is 
orientations preserving (resp. reversing), where $N= Z_{\infty}$. 
It follows from \eqref{eqn:F-inf-form*-1} for $N$ that $\epsilon_1 = (-1)^{d-p}$. 
We now note that $(F_{\infty})_* = \epsilon_2 (F_{\infty}^{-1})^*$,
where $\epsilon_2 = \pm 1$ depending if $F_{\infty}\colon X_{\infty} \to X_{\infty}$ is orientation preserving or reversing. By \eqref{eqn:F-inf-form*-1} for $X_{\infty}$, we have $\epsilon_2 =  (-1)^d$. Combining these observations with \eqref{ex:real-current-delta-Z}, we have $F_{\infty}^* \delta_Z = (-1)^p \delta_Z$ and hence $\delta_Z \in D^{p,p}(X_{\R})$. 

\item For a rational function $f \in k(X)^{\times}$  on $X$, the induced current $[-\log|f^2|] \in D^{0, 0}(X_{\infty})$ is a real current that belongs to the subgroup $D^{0, 0}(X_{\R})$. The main observation one needs to prove the above claim is that $F_{\infty}^* (f_{\sigma}) = \ov{f_{\Theta(\sigma)}}$, where $\Theta$ is an in Lemma 
\ref{lem:F-infty-C-Sigma} and $f_{\sigma} \in k(X_{\sigma})^{\times}$ be the pull-back of $f$ via the morphism $\sigma \colon X_{\sigma} \to X$.

\item If $k = \C$ then from Example \ref{ex:F-infty} it follows that 
\[
D^{p,p}(X_{\R}) = \{T\in D^{p, p}(X(\C))| \textnormal{ $T$ is real}\}.
\]
In particular, in this case, we don't have to worry about $F_{\infty}$-action on currents and we can work with real currents. 
\end{enumerate}
\end{ex}

\subsection{Arithmetic Chow groups}  \label{sec:Def-ACG}

Let $k = (k, \Sigma, F_{\infty})$ be an arithmetic field and let $f \colon X \to \Spec(k)$ be a smooth quasi-projective $k$-scheme of dimension $d \geq 0$. We let other notations be as in previous sections. Let $\wt{D}^{p,p}(X_{\R})$ to be the image of $D^{p,p}(X_{\R})$ inside $\widetilde{D}^{p,p}(X_{\infty})$.

Recall that a real current $g_{Z}\in D^{p-1,p-1}(X_{\infty})$ is said to be a Green current for a codimension $p$ cycle $Z$ if $\delta_Z + d d^c g_{Z} = [\omega_{Z}]$ such that $\omega_{Z} \in A^{p, p}(X_{\infty})$. It follows from 
\cite[Theorem~1.2.1]{GS90} that for a cycle $Z$, a Green current $g'_Z$ always exists. Taking $g_{Z} = \frac{1}{2}(g'_{Z} + (-1)^{p-1} F_{\infty}^*g'_{Z})$, we get a Green current $g_{Z} \in {D}^{p-1,p-1}(X_{\R})$, and \cite[Lemma~1.2.4]{GS90} shows that two different Green currents for the same cycle differs by an element in $\widetilde{A}^{p-1,p-1}(X_{\R})$.

We define the arithmetic Chow groups as follows: Let 
\begin{equation}  \label{eqn:ACG-cyc}
\wh{\caZ}^p(X) := \{(Z, \widetilde{g}_{Z}) \in \caZ^p(X) \oplus \wt{D}^{p-1,p-1}(X_{\R})~|~g_{Z}~\text{a Green current for}~Z\}. 
\end{equation}
For a closed subvariety $W$ of codimension $p-1$ and $f \in k(W)^{\times}$, we let $ \wh{{\rm div}(f)} := ({\rm div}(f), (\iota_W)_* [-\log | f |^2])$, where $\wt{W} \to W$ is a resolution of singularities, $\iota_W \colon \wt{W} \to X$ denote the induced map of smooth $k$-schemes and $(\iota_W)_* \colon D^{0, 0}(\wt{W}_\R) \to D^{p-1, p-1}(X_{\R})$ is the push-forward map of currents. It follows from \eqref{eqn:Div-An*-0} that 
$\wh{{\rm div}(f)} \in \wh{\caZ}^p(X)$. We denote the subgroup generated by all such $\widehat{{\rm div}(f)}$ by $\wh{\caR}^{p}(X)$, and define the arithmetic Chow group of $X$ as the quotient
\begin{equation} \label{eqn:ACG-def}
\wh{\CH}^p(X) := \frac{\wh{\caZ}^p(X)}{\wh{\caR}^p(X)}.    
\end{equation}

The existence of a Green current for an algebraic cycle shows that the canonical map $\widehat{\CH}^{p}(X)\rightarrow \CH^{p}(X)$, which maps $[(Z,\widetilde{g}_{Z})]$ to $[Z]$, is surjective. This surjective map fits in an exact sequence, which we recall among other properties of arithmetic Chow groups (see \cite[Section~3]{GS90} for details):
\begin{itemize}
\item
The arithmetic Chow groups fit into an exact sequence
\begin{displaymath}\label{eq:1001}
  \CH^p(X,1)\rightarrow\widetilde{A}^{p-1,p-1}(X_{\R})
  \rightarrow\widehat{\CH}^p(X)\rightarrow
  \CH^p(X)\rightarrow 0,  
\end{displaymath}
where $\CH^p(X,1)$ is Bloch's higher Chow group.
\item
There is an intersection pairing
\begin{displaymath}
\widehat{\CH}^p(X)\otimes \widehat{\CH}^q(X)\rightarrow
\widehat{\CH}^{p+q}(X),
\end{displaymath}
which turns $\oplus_{p\geq 0}\widehat{\CH}^p(X)$ into a unitary graded commutative algebra.
\item
Let $f\colon X\rightarrow Y$ be a smooth morphism of smooth quasi-projective
varieties. Then there exists a pullback
$$f^\ast\colon \widehat{\CH}^p(Y)\rightarrow \widehat{\CH}^p(X),$$
and if $f$ is also proper, then there is a pushforward map
$$f_{\ast}\colon \widehat{\CH}^p(X)\rightarrow \widehat{\CH}^{p-e}(Y),$$
where $e=\dim(X)-\dim(Y)$. 
The pullback is an algebra
homomorphism, and together with the pushforward, it satisfies the
projection formula.
\item The intersection product, the direct image, and the inverse image
  are compatible with the corresponding operations on the classical
  Chow groups. 
\end{itemize}
In \cite[Section~5.2]{GS90-2} it was shown that the codimension one arithmetic Chow group is isomorphic to the arithmetic Picard group of Hermitian line bundles on $X$. We briefly recollect the morphisms that define this isomorphism in the following subsection.

\subsection{Picard group of Hermitian line bundles} \label{sec:APG}

We start this section by defining the Picard group of Hermitian line bundles.
We let $X$ denote a smooth projective variety over an arithmetic field $k = (k, \Sigma, F_{\infty})$.  

Let $\sL$ be a line bundle on $X$. We let $(U_{\alpha})_{\alpha \in I}$ denote a cover of $X$ that trivialises $\sL$ and let $g_{\alpha \beta} \in \sO(U_{\alpha\beta})$ denotes the transition function of $\sL$, where $U_{\alpha\beta} = U_{\alpha} \cap U_{\beta}$. 
Let  $\sL_{\infty}$ denote the associated line bundle on $X_{\infty}$. 
A smooth Hermitian metric $||\cdot ||$ on $\sL_{\infty}$ is given by the data 
$(U_{\alpha,\infty}, h_{\alpha,\infty})_{\alpha \in I} := (U_{\alpha, \sigma}, h_{\alpha, \sigma})_{\sigma \in \Sigma, \alpha \in I}$, where 
$h_{\alpha, \sigma} \colon U_{\alpha, \sigma}(\C)\to \R_{>0}$ are smooth functions such that for all $\alpha, \beta \in I$  and $\sigma \in \Sigma$, we have
\begin{equation}
    h_{\alpha, \sigma} = |\sigma^*(g_{\beta \alpha})|^2 h_{\beta, \sigma} \textnormal{ on $U_{\alpha \beta, \sigma}$ and } F_{\infty}^*(h_{\alpha,\infty}) = h_{\alpha,\infty}. 
\end{equation}
By a Hermitian line bundle $(\sL, ||\cdot||)$ on $X$, we mean an algebraic line bundle $\sL$ on $X$ and a smooth Hermitian metric $||\cdot ||$ on $\sL_{\infty}$. In the above notations (i.e., given a cover $(U_{\alpha})_{\alpha \in I}$ that trivialises $\sL$), we write the Hermitian line bundle $(\sL, ||\cdot||)$ by $(U_{\alpha}, g_{\alpha \beta}, h_{\alpha, \sigma})$. We let $\wh{\Pic}(X)$ denote the set of isomorphic classes of  Hermitian line bundles on $X$.

For Hermitian line bundles $(\sL_1, || \cdot ||_1)  = (U_{\alpha}, g^1_{\alpha \beta}, h^1_{\alpha, \sigma})$ and $(\sL_2, || \cdot ||_2)  = (U_{\alpha}, g^2_{\alpha \beta}, h^2_{\alpha, \sigma})$ on $X$, we define their product as follows.
\[
(\sL_1, || \cdot ||_1)  \cdot (\sL_2, || \cdot ||_2) :=  (U_{\alpha},  g^1_{\alpha \beta} g^2_{\alpha \beta}, h^1_{\alpha, \sigma} h^2_{\alpha, \sigma}).
\]
Note that we can always assume both $\sL_i$ have trivialization for an open cover of $X$ and it is easy to check that all the above definitions do not depend on the cover.  
 With the above product, the set $\wh{\Pic}(X)$ forms an abelian group. Note that the identity element of $\wh{\Pic}(X)$ is given by the class of the  trivial Hermitian line bundle 
$(\sO_X, |\cdot |)= (X, 1, c_1)$, where $1 \in \sO(X)$ is the identity element and $c_1 \colon X_{\infty} \to \R_{>0}$ is the constant function with value $1$ and the inverse of an element 
$(\sL, ||\cdot||)=(U_{\alpha}, g_{\alpha \beta}, h_{\alpha, \sigma})$ is given by 
$-(\sL, ||\cdot||) := (U_{\alpha}, 1/g_{\alpha \beta}, 1/h_{\alpha, \sigma})$.

\begin{df} \label{defn:coho-chern-class}
Given a Hermitian line bundle $(\sL, ||\cdot||)$ on $X$ as above, we define its cohomological  Chern class 
$c_1(\sL, ||\cdot||)  \in {A}^{1,1}(X_{\infty}) $ so that  
\begin{equation}\label{eqn:Coho-Chern-class*-1}
c_1(\sL, ||\cdot||)|_{U_{\alpha, \infty}} = \frac{1}{2\pi i} \partial \overline{\partial}\log  h_{\alpha, \infty} \in A^{1,1}(U_{\alpha, \infty}). 
\end{equation}
Since the transition functions $g_{\alpha\beta} \colon U_{\alpha \beta, \infty} \to \C$ are nowhere vanishing holomorphic functions, it follows from \eqref{eqn:Div-An*-0} that the right-hand side in the above expression agrees on $U_{\alpha \beta}$. 

We therefore have a well define element $c_1(\sL, ||\cdot||)  \in A^{1,1}(X_{\infty}) $.
Since each $h_{\alpha, \sigma}$ are real valued $C^{\infty}$-function and $F_{\infty}^*(h_{\alpha,\infty}) = h_{\alpha,\infty}$, it follows that $c_1(\sL, ||\cdot||) \in A^{1,1}(X_{\R})$. 
\end{df}

We now define the arithmetic Chern class of a Hermitian line bundle. 
Recall that a rational section $s$ of $\sL$ is given by sections $s_{\alpha} \in \sO(U_{\alpha} \cap U)$ such that $s_{\alpha} = g_{\alpha \beta} s_{\beta}$ on $U \cap U_{\alpha \beta}$, where $U_{\alpha \beta} = U_{\alpha} \cap U_{\beta}$ and $U\subset X$ is a dense open subset. For short, we shall write $s = (U, U_{\alpha}, s_{\alpha})$ for such a section of the line bundle $\sL = (U_{\alpha}, g_{\alpha \beta})$. Let $s = (U, U_{\alpha}, s_{\alpha})$ be a rational section of the Hermitian line bundle $(\sL, || \cdot||) = (U_{\alpha}, g_{\alpha \beta}, h_{\alpha})$. We define its norm as 
\[
||s||^2 := s_{\alpha} \overline{s_{\alpha}} h_{\alpha}.
\]
Note that the above expression does not depend on $\alpha$. 
Let $V \subset U$ be the dense open subset where $s$ does not vanish (i.e., $V = U \setminus \textnormal{div}(s)$). The continuous map $ \log  || s||^2 \colon V_{\infty} \to \R$ then defines an element of $L^1_{\textnormal{loc}}(X_{\infty})$ and hence defines a current  $[\log  || s||^2] \in  D^{0,0}(X_{\infty})$. In fact, we have  $ [\log  || s||^2] \in D^{0,0}(X_{\R})$ and
\begin{equation}\label{eqn:A-Chern-class*-1}
  dd^c (-[ \log  || s||^2])+ \delta_{{\rm div}(s)} = [c_1(\sL, ||\cdot||) ].
 \end{equation}
In other words, we have $( \Div(s), -[ \log  || s||^2] ) \in \widehat{\mathcal{Z}}^1(X)$.

\begin{df} \label{defn:A-Chern-class}
For a Hermitian line bundle $(\sL, || \cdot||) = (U_{\alpha}, g_{\alpha \beta}, h_{\alpha})$ on $X$, we define its arithmetic Chern class as follows:
\begin{equation} \label{eqn:A-Chern-class*-0}
\widehat{c}_1(\sL, || \cdot ||) := ( \Div(s), -[ \log  || s||^2] ) \in \widehat{\CH}^1(X).
\end{equation}
 \end{df}
Since two sections of a line bundle differ by a rational function, it follows that the above element does not depend on the chosen section $s$. The following proposition is well known, and we provide a local description of the maps.
\begin{prop}\label{prop:A-Chern-class-3}
The assignment $(\sL, || \cdot ||)\mapsto \widehat{c}_1(\sL, || \cdot ||)$ defines a group isomorphism
\begin{equation}\label{eqn:A-Chern-cl-3.1}
\widehat{c}_1 \colon \wh{\Pic}(X) \to \widehat{\CH}^1(X). 
\end{equation}
\end{prop}
\begin{proof}
We construct an explicit inverse map that we call the cycle class map. 
Let $(Z, g) \in \wh{\mathcal{Z}}^1(X)$ be a cycle. Since $X$ is smooth, $Z$ is a Cartier divisor $(U_{\alpha}, f_{\alpha})$, where $f_{\alpha}$ is a rational function on $ U_{\alpha}$ such that $g_{\alpha \beta}= f_{\alpha}/ f_{\beta}\in \sO_{U_{\alpha \beta}}^{\times}$. We take the associated line bundle 
$\sO_X(Z) = (U_{\alpha}, g_{\alpha \beta}= f_{\alpha}/ f_{\beta})$. We now give the Hermitian inner product on $\sO_X(Z)$. For this, we note that two Green currents of a cycle differ by a smooth form (see \cite[Lem.~1.2.4]{GS90}). By  \eqref{eqn:A-Chern-class*-1}, we have 
\begin{equation} \label{eqn:A-Chern-class-3-0*}
i/2 \pi \partial \ov{\partial} ([ - \log |\sigma^*(f_{\alpha}) |^2]) = - \delta_{Z \cap U_{\alpha, \sigma}} = 
i/2 \pi \partial \ov{\partial}(g)|_{U_{\alpha, \sigma}} \in D^{1,1}(U_{\alpha, \sigma}). 
\end{equation}
In particular, there exist $h_{\alpha, \sigma}' \in A^{0,0}(U_{\alpha, \sigma})$ such that $
[h_{\alpha, \sigma}'] = (g)|_{U_{\alpha, \sigma}} + [\log |\sigma^*(f_{\alpha})|^2]$. 
We let $h_{\alpha, \sigma} = {{\rm exp}}^{- h_{\alpha, \sigma}'} \colon U_{\alpha, \sigma} \to \R_{>0}$. Then $h_{\alpha, \sigma}$ are smooth real functions that satisfy 
\begin{equation}
 |g_{\beta \alpha}|^2 h_{\beta, \sigma}  = h_{\alpha, \sigma} \textnormal{ and } F_{\infty}^*(h_{\alpha, \infty}) = h_{\alpha, \infty},
\end{equation}
where $h_{\alpha, \infty} \colon U_{\alpha, \infty} \to \R_{>0} $ is the induced smooth function. 
We therefore have a Hermitian line bundle 
\begin{equation} \label{eqn:A-Chern-class-3-1*}
(\sO_X(Z), || \cdot||_g) := (U_{\alpha}, g_{\alpha \beta}= f_{\alpha}/f_{\beta}, h_{\alpha, \sigma} = {{\rm exp}}^{-(g)|_{U_{\alpha, \sigma}} - [\log |\sigma^*(f_{\alpha})|^2]}).
\end{equation}
The assignment $(Z, g) ) \mapsto (\sO_X(Z), || \cdot||_{g})$ gives 
 well-defined group homomorphism 
\begin{equation} \label{eqn:A-Chern-class-3-4*}
\cyc_X \colon  \widehat{\CH}^1(X) \to \wh{\Pic}(X).
\end{equation}
By unfolding the definitions, we see that 
$\wh{c}_{1} \circ \cyc_X = {\rm Id}$  and $
\cyc_X  \circ \wh{c}_{1} = {\rm Id}$. 
This implies the arithmetic Chern class map 
\eqref{eqn:A-Chern-cl-3.1} is an isomorphism.  
\end{proof}

\section{Chow groups with modulus} \label{sec:CH-M}

In this section, we recall the definition and basic properties of the Chow group with modulus that was defined in \cite[\S3]{Binda-Saito} (the name relative Chow group was used in loc.cit.).

We fix the following notations: Let $k$ be a field and let $X$ be a smooth quasi-projective scheme over $k$ of pure dimension $d \geq 0$. We let $D \hookrightarrow X$ be an 
effective Cartier divisor on $X$. We call the pair $(X, D)$ to be a modulus pair.

We let $\mathcal{Z}^p(X|D)$ denote the free abelian group generated on the set of codimension  $p$ irreducible closed subsets $Y$ 
of $X$, such that $Y \cap D = \emptyset$.  
%For $x\in U^{(i)}$, we let $[\overline{\{x\}}]$ denote the generator defined by $x$. 
For a codimension $(p-1)$ irreducible closed subset $W$ such that $W\not\subset D$, we let $\nu: W^N \to X$ denote the projection from the normalization of 
$W$ and let $E = \nu^*(D)$ denote the pull-back of $D$ to $W^N$. We let $G(W^N, E)$ denote the set of non-zero rational functions on $W^N$ that are regular (actually a unit) in a neighborhood of $E$ and it is one along $E$. More precisely:
\begin{df}\label{eqn:CH-M*-1-0}
We define $G(W^N, E)$ as follows.
\begin{align*}
G(W^N, E)\coloneqq \varinjlim_{E \subset V \subset W^N} \Gamma(V, {\rm Ker}(\sO_V^{\times} \rightarrow \sO_E^{\times})) \subset k(W^N)^{\times}\\
=\{f\in k(W^{N})^{\times}\colon f\in \ker(\mathcal{O}^{\times}_{W^{N},y}\to (\mathcal{O}_{W^{N},y}/I_{E,y}\mathcal{O}_{W^{N},y})^{\times})~\text{for any}~ y\in E\},
\end{align*}
where $I_{E,y}$ is the ideal of $\mathcal{O}_{W^{N},y}$ defining $E$ at $y$. 
\end{df}
We note that if $f \in G(W^N, E) $, then ${\rm div}_{W^N}(f) \in \mathcal{Z}^1(W^N| E)$. If $w$ is the generic point of $W$, we let 
$
\partial_{w}: G(W^N, E) \to \mathcal{Z}^{p}(X| D)
$
 denote the composition of the maps 
 \begin{equation} \label{eqn:CH-M*-1-2}
  G(W^N, E) \xrightarrow{{\rm div_{W^N}}} \mathcal{Z}^1(W^N| E) \xrightarrow{\nu_*} \mathcal{Z}^{p}(X | D),
  \end{equation}
  where the second arrow is the push-forward map along the finite morphism 
  $W^N \setminus E \to U := X \setminus D$. 
We let $\mathcal{R}^{p}(X|D)$ denote the image of the induced  map
\begin{displaymath}
\oplus_{w \in U^{(p-1)}} G(W^N,E) \xrightarrow{\partial_w} \mathcal{Z}^p(X|D),
\end{displaymath}
 where $W = \overline{\{w\}} \subset X$.

\begin{df}\label{defn:CH-Modulus}
We define the Chow group $\CH^p(X|D)$ of codimension $p$ cycles with modulus  as the quotient
\begin{equation} \label{eqn:CH-M*-1-3}
\CH^p(X|D) := \frac{\mathcal{Z}^p(X|D)}{\mathcal{R}^p(X|D)}.
\end{equation}
For $p \geq 0$, we define the Chow group of dimension $p$ cycles with modulus $\CH_p(X|D)$ of 
the pair $(X, D)$ as
\begin{equation} \label{eqn:CH-M*-1-4}
\CH_p(X|D) := \CH^{d-p}(X|D),
\end{equation}
where $d$ is the pure-dimension of $X$. 
\end{df}

\vskip .4pc

\begin{rmk}
It follows from the above definition that these groups agree with Fulton's Chow groups $\CH^p(X)$ when $D =\emptyset$, i.e.,
$\CH^p(X|\emptyset) = \CH^p(X)$. Following Bloch's higher Chow groups, a general theory of the higher Chow groups with modulus $\CH^p(X|D, m)$ was introduced in \cite{Binda-Saito}. They also showed that the groups $\CH^p(X|D, 0)$ agree with the above groups $\CH^p(X|D)$ (see Theorem~3.3 of loc.cit.).
\end{rmk}

We now list the functoriality properties of the Chow groups with modulus with proper references. First we define
\begin{df}\label{df:modpair}
By a morphism $f\colon (Y, E) \to (X, D)$ of modulus pairs, we mean a morphism $f \colon Y \to X$ of smooth quasi-projective varieties such that $f^*(D)$ is a divisor on $Y$ and $f^*(D) \leq E$.
\end{df}

\begin{prop} \label{prop:functmodulus}
Chow groups with modulus satisfy the following. 
\begin{enumerate}

\item\cite[Proposition~2.10]{KP17} Let $f \colon (Y, E) \to (X, D)$ be a proper morphism of modulus pairs. Then for all $j \geq 0$, there is a natural push-forward map 
\[
f_* \colon \CH_j (Y|E) \to \CH_j (X|D)
\]
such that $(g \circ f)_* = g_* \circ f_*$, where $f$ and $g$ are composable proper morphisms of modulus pairs. 

\item\cite[Proposition~2.12]{KP17} Let $f \colon (Y, E) \to (X, D)$ be a flat morphism of relative dimension $d \geq 0$, then there is a well defined pull-back map 
\[
f^* \colon \CH_j(X,D) \to \CH_{j+d}(Y, E)
\]
when $E = f^*(D)$. Moreover, $(g \circ f)^* = f^* \circ g^*$, where $f$ and $g$ are composable flat morphisms of modulus pairs. 

\subsection{Relative Picard Group and cycles  with modulus of codimension $1$} \label{sec:Rel-Pic}
\end{enumerate}
\end{prop}

We continue with the above notations. In this section, we study the Chow group with modulus  $\CH^1(X|D)$ of codimension $1$-cycles. 
Recall that the Chow group $\CH^1(X)$ is isomorphic to the group $\Pic(X)$ of line bundles on $X$. 
We recall its relative version from 
\cite[Sec. 2]{Suslin-Voe}. Compare the results of this section with that of Section \ref{sec:APG}. 

\begin{df}\label{def:Rel-Pic}
For a line bundle $\sL$ on $X$, by a trivialization of $\sL$ along $D$, we mean an isomorphism $\phi\colon \sL_{|D} \xrightarrow{\cong} \sO_D$ that lifts to an open neighborhood of $D$, i.e., there exists an open subset $D \subset V \subset X$ and an isomorphism $\phi' \colon \sL_{|V} \xrightarrow{\cong} \sO_V$ such that 
$\phi = \phi'_{|D}$. 
We define the relative Picard group as follows. 
\begin{equation} \label{eqn:Rel-Pic*-0}
\Pic(X,D) = 
\{(\sL, \phi\colon \sL_{|D} \xrightarrow{\cong} \sO_D) | \textnormal{ $\phi$ is a trivialization of $\sL$ along $D$}\}.
\end{equation}
\end{df} 

Note that here by $(\sL, \phi)$, we mean the isomorphic class of such a pair. For $(\sL_1, \phi_1), (\sL_2, \phi_2) \in \Pic(X, D)$, we define their product as follows. 
\begin{equation} \label{eqn:prod-RPic}
(\sL_1, \phi_1)\cdot (\sL_2, \phi_2) := (\sL_1 \otimes \sL_2, \phi_1\otimes \phi_2) \in \Pic(X, D),    
\end{equation}
where $ \phi_1\otimes \phi_2 \colon (\sL_1 \otimes \sL_2)_{|D} \cong (\sL_1)_{|D} \otimes (\sL_2)_{|D} \xrightarrow{\cong} \sO_D$ is the induced isomorphism. 

For a trivialization $\phi\colon \sL_{|D} \xrightarrow{\cong} \sO_D$, we let $\phi^{\vee} \colon \sO_D \xrightarrow{\cong}  \sL^{-1}_{|D}$ denote the isomorphism induced by tensoring $\phi$ with the dual line bundle $ \sL^{-1}_{|D}$. 
With the product structure \eqref{eqn:prod-RPic}, $\Pic(X,D)$ forms a group with identity $(\sO_X, \id \colon (\sO_X)_{|D} \xrightarrow{\cong} \sO_D)$ and $(\sL, \phi)^{-1}= (\sL^{-1}, (\phi^{\vee})^{-1})$. 

Let $(\sL, \phi) \in \Pic(X, D)$ and let $D \subset V_1, V_2\subset X$ be are two open neighborhoods of $D$ with isomorphisms 
$\phi_i' \colon \sL_{|V_i} \xrightarrow{\cong} \sO_{V_i}$ such that 
$\phi_i = (\phi_i')_{|D}$. Then $V = V_1 \cap V_2$ is an open neighbourhood of $D$ and $s_i = (\phi_i')_{|V}^{-1}$ are rational sections of $\sL$ that are regular on $V$ and $(s_1)_{|D} = (s_2)_{|D}$. Since two (regular) sections of a line bundle differ by a rational (resp. regular) function, there exists $f\in \sO_X(V) \subset k(X)^{\times}$ such that $s_1 = f s_2$ and $f \in G(X, D)$. It therefore follows that 
\begin{equation} \label{eqn:RPic-div}
\Div(s_1) = \Div(s_2) \in \CH^1(X|D).
\end{equation}
In particular, the following definition makes sense. 
\begin{df} \label{def:c1-M}
Given $(\sL, \phi) \in \Pic(X, D)$, we define its Chern class with modulus as follows. 
\begin{equation} \label{eqn:c1-M*-1}
c_1(\sL, \phi) := \Div((\phi')^{-1}) \in \CH^1(X|D), 
\end{equation}
where $\phi' \colon \sL_{|V} \xrightarrow{\cong} \sO_V$ is an isomorphism for an open neighbourhood $V$ of $D$  such that 
$\phi = \phi'_{|D}$. 
\end{df}

\begin{prop} \label{prop:c1-M-Iso}
    The assignment $(\sL, \phi) \mapsto c_1(\sL, \phi)$ defines a group isomorphism
\begin{equation} \label{eqn:c1-M*-2}
c_1\colon \Pic(X, D) \xrightarrow{\cong} \CH^1(X|D).
\end{equation}
\end{prop}
\begin{proof}
It follows from a easy calculation that the assignment $(\sL, \phi) \mapsto c_1(\sL, \phi)$ defines a group homomorphism 
$c_1\colon \Pic(X, D) \to \CH^1(X|D)$. We define the inverse map $\cyc_{X|D} \colon \CH^1(X|D) \to \Pic(X, D)$ that we call the cycle class map with modulus.

Let $Z\in \mathcal{Z}^1(X|D)$. Since $X$ is a smooth variety, the Weil divisors on $X$ are the same as the Cartier divisor on $X$. 
Let $  (U_{\alpha}, f_{\alpha})$ be a Cartier divisor for $Z$ (i.e., 
$Z \cap U_{\alpha} = \Div(f_{\alpha})$).
 Let $\sO(U_{\alpha}, f_{\alpha})$ be a line bundle in the isomorphic class $(U_{\alpha}, g_{\alpha} = f_{\alpha}/f_{\beta})$ of line bundles on $X$. 
 If we let $U = X \setminus |Z|$, then $f_{\alpha} \in \sO(U \cap U_{\alpha})^{\times}$ and they patch up to yield a section $s \in \Gamma(U, \sO(U_{\alpha}, f_{\alpha}))$ such that $s \colon \sO_{U}\xrightarrow{\cong} \sO(U_{\alpha}, f_{\alpha}))_{|U}$. We let $\phi=(s_{|D})^{-1} \colon \sO(U_{\alpha}, f_{\alpha}))_{|D} \xrightarrow{\cong} \sO_{D}$ be the induced trivialisation of the line bundle  $\sO(U_{\alpha}, f_{\alpha})$. We therefore have an element 
$( \sO(U_{\alpha}, f_{\alpha}), \phi=(s_{|D})^{-1} ) \in {\Pic}(X, D)$. Moreover, one can show that the above element does not depend on the choice of the Cartier divisor associated with $Z$ and yields a well-defined group homomorphism 
\begin{equation} \label{eqn:c1-M-Iso*-1}
\cyc_{X|D} \colon  \CH^1(X|D) \to  \Pic(X, D)
\end{equation}
such that $\cyc_{X|D} (Z) =  (\sO(U_{\alpha}, f_{\alpha}), \phi=
(s_{|D})^{-1} ) \in {\Pic}(X, D)$ for $Z \in \mathcal{Z}^1(X|D)$. It follows from the definition of the maps that $\cyc_{X|D}$ is the inverse of the map in \eqref{eqn:c1-M*-2}. We skip the details for these checks from here because we write the details later on in a more general setting (see the proof of Theorem \ref{thm:Chern-Iso}). 
\end{proof}

\subsection{Action of Chow ring} \label{sec:Action-CR}
We recall the following theorem from \cite[Theorem~1.1]{KP17} and sketch a proof. 

\begin{thm} \label{thm:Prod-CGM}
    Let $X$ be a smooth quasi-projective variety over a field $k$ and let $D$ be an effective Cartier divisor on $X$. Then we have a well-defined cap product 
    \begin{equation} \label{eqn:Prod-cycle*-0}
        \cap \colon \CH^p(X) \times \CH^q(X|D) \to \CH^{p+q}(X|D)
    \end{equation}
    so that $\CH^*(X|D)$ becomes module over the ring $\CH^*(X)$. 
\end{thm}
    We sketch a proof here because in \cite{KP17}, a more general statement is proven. 
    Recall that we say two cycles $\alpha = \sum_l n_l [Y_l] \in \mathcal{Z}^{p}(X)$ and $\beta = \sum_t m_t [Z_t] \in \mathcal{Z}^{q}(X)$ intersect properly if for each $l, t$, the intersection $Y_l \cap Z_t$ is either empty or every irreducible component of $Y_l \cap Z_t$ is of codimension $p+q$. 
    In this case, we have a well-defined cycle 
    $\alpha\cap\beta = \sum_{l, t} n_l m_t [Y_l\cap Z_t] \in  \mathcal{Z}^{p+q}(X)$. Moreover, if $\beta \in \mathcal{Z}^{q}(X|D)$, then $\alpha\cap\beta \in  \mathcal{Z}^{p+q}(X|D)$ because 
    $Y_l \cap Z_t \cap D \subset Z_t \cap D = \emptyset$ for all $l, t$. The following lemma is key to proving Theorem~\ref{thm:Prod-CGM}.

    \begin{lem}\label{lem:Prod-cycle-2}
    Let $\alpha \in \mathcal{R}^p(X)$ and 
    $\beta \in \mathcal{Z}^{q}(X|D)$ be such that they intersects properly. Then $\alpha \cap \beta \in  \mathcal{R}^{p+q}(X|D)$. 
    \end{lem}
    \begin{proof}
     By the moving lemma for $K_1$-chains \cite[Lemma~4.2.6]{GS90}, we can find finitely many closed subvarieties $W_l$ of $X$ and rational functions $h_l$ on $W_l$ such that 
        $\alpha= \sum_l m_l {\Div(h_l)}$ and that each $\Div(h_l)$ intersects with $\beta$ properly. 
       We can therefore assume that $\alpha = \Div(h)$, where $h\in k(W)^{\times}$ for a closed subvariety $W$ of codimension $p-1$ and $\beta = [Z]$, where $Z$ is a closed subscheme of $X$ that interests with each irreducible components of $\Div(h)$ properly. 
        
        Let $W \cap Z = S \cap T$, where $S$ is the union of irreducible components of codimension $p+q-1$ and $T$ is a union of all other irreducible components. Since $\Div(h)$ intersects with $Z$ properly, $\Div(h) \cap T = \emptyset$. In other words, $h|_{T}$ is a regular function on $T$ that is a unit on $T$. By \cite[Propositions~8.1.1 and 8.2]{Fulton} (also see \cite[Page~141]{GS90}), we know that 
    $[W] \cdot [Z] = \sum_l m_l S_l + \gamma \in \CH^{p+q-1}(X)$, where $S_l$ are irreducible components of $S$ and $\gamma \in \CH_T^{p+q-1}(X)$. Let 
    $\gamma$ is represented by an element $ \sum_t n_t [T_t] \in
    \mathcal{Z}^{p+q-1}_T(X)$. Since $h|_T$ is a unit, we have an element 
    \begin{equation}\label{eqn:Prod-cycle*-6}
    [h] \cdot \sum_t n_t [T_t] := \sum_t n_t [h|_{T_t}] \in {\rm Ker}(\oplus_{x\in X^{(p+q-1)}\cap T} k(x)^{\times} \to \oplus_{x\in X^{(p+q)}\cap T} \mathbb{Z}).
    \end{equation}
    We let $[h]\cdot \gamma$ be the image of the above element in
    \begin{displaymath}
    \CH^{p+q, p+q-1}_T(X) = \frac{{\rm Ker}(\oplus_{x\in X^{(p+q-1)}\cap T} k(x)^{\times} \to \oplus_{x\in X^{(p+q)}\cap T} \mathbb{Z})}{{\rm Im}(\oplus_{x\in X^{(p+q-2)}\cap T} K_2(k(x)) \to \oplus_{x\in X^{(p+q-1)}\cap T} k(x)^{\times})},
    \end{displaymath}
    and define
    \begin{equation}\label{eqn:Prod-cycle*-7}
    h. Z = \sum_{l}  m_l [h|_{S_{l}}] + [h]\cdot \gamma \in \CH^{p+q, p+q-1}(X). 
    \end{equation}
    Note that the above element $h.Z$ lies in the image of the group 
    $\CH^{p+q, p+q-1}_Z(X)$ under the forget support map. 
    It then follows from \cite[Lemma~4.2.5(1)]{GS90} that 
     \begin{equation}\label{eqn:Prod-cycle*-8}
     \alpha \cdot \beta =  \Div(h) \cdot [Z] = \Div (h. Z). 
    \end{equation}
   By the definition of $\mathcal{R}^*(X|D)$, it 
    suffices to show that the element $h.Z$ is represented by an element of 
   $\oplus_{x\in X^{(p+q-1)}} G(Y^N, E) \subset \oplus_{x\in X^{(p+q-1)}} \subset k(x)^{\times}$, where $Y$ is the closure of $x$ in $X$. But this follows because $h.Z$ is represented by 
   the element $\sum_{l}  m_l [h|_{S_{l}}] +  \sum_t n_t [h|_{T_t}] $ and for each $l, t$, we have  
   $S_l \cap D \subset Z\cap D = \emptyset$ and $T_t \cap D\subset T \cap D \subset Z \cap D = \emptyset$. This completes the proof. 
    \end{proof}

\subsection{Proof of Theorem~\ref{thm:Prod-CGM}}

 Let $\beta \in \mathcal{Z}^q(X|D)$ and let $\ov{\alpha} \in \CH^p(X)$. By Chow's moving lemma \cite[\S~3, Theorem]{Roberts}, 
    %(also see \cite[Lemma 0B1U]{StackP}, 
    we can find a cycle $\alpha_0 \in \mathcal{Z}^p(X)$ such that $\ov{\alpha} = \ov{\alpha_0}$ and $\alpha_0$ intersects with $\beta$ properly. In particular,  we have a well-defined element 
    $\alpha_0 \cap \beta \in \mathcal{Z}^{p+q}(X|D)$. 
    
    We define the cap 
    product $\ov{\alpha} \cap \beta \in \CH^{p+q}(X|D)$ to be the image of the cycle 
    $\alpha_0 \cap \beta$. We first note that the element $\ov{\alpha} \cap \beta \in \CH^{p+q}(X|D)$ does not depend on the choice of $\alpha_0$. Indeed, if $\alpha_0'$ is another cycle so that $\ov{\alpha} = \ov{\alpha_0'}$ and $\beta$ interests with $\alpha_0$ properly, then $\alpha_0 - \alpha_0' = \gamma \in \mathcal{R}^p(X)$ 
    and  $\beta$ intersects with $\gamma$ properly. By Lemma \ref{lem:Prod-cycle-2}, we have 
    \begin{equation}\label{eqn:Prod-cycle*-9}
       \alpha_0 \cap \beta - \alpha_0' \cap \beta =  (\alpha_0 - \alpha_0') \cap \beta = \gamma \cap \beta \in \mathcal{R}^{p+q}(X|D).
    \end{equation}
     This proves that the cycle class $\ov{\alpha} \cap \beta \in \CH^{p+q}(X|D)$ is well defined. 
     Note that for all $\beta \in \mathcal{Z}^q(X|D)$, we now have a well-defined group homomorphism $- \cap \beta \colon \CH^p(X) \to \CH^{p+q}(X|D)$. Since the intersection product of cycles is linear in both coordinates, we get a bilinear pairing 
    $\cap \colon \CH^p(X) \times \mathcal{Z}^q(X|D) \to \CH^{p+q}(X|D)$. We now show that $\CH^p(X) \cap \mathcal{R}^q(X|D) = 0 \in \CH^{p+q}(X|D)$. 

    Let $\beta = \Div(f) \in \mathcal{R}^q(X|D)$, where $f$ is a rational function on a integral closed subscheme $W$ and let $\ov{\alpha} \in \CH^p(X)$. We need to show that $\ov{\alpha} \cap \beta = 0 \in \CH^{p+q}(X|D)$. By Chow's moving lemma, we can find a cycle $\alpha_0  \in \mathcal{Z}^p(X)$ such that $\ov{\alpha_0} = \ov{\alpha} \in \CH^p(X)$, and it intersects with $W$, $W \cap D$ and support of $\Div(f)$ properly. Let $\alpha_0 \cap [W]= \sum_l m_l [Z_l] \in \mathcal{Z}^{p+q-1}(X)$.  We then have
    \begin{equation} \label{eqn:Prod-cycle*-10}
        \ov{\alpha} \cap \beta = \alpha_0 \cap \Div(f)= \sum_l m_l \Div(f|_{Z_l}) \in \CH^{p+q}(X|D). 
    \end{equation}
    It therefore suffices to show that $f|_{Z_l}$ satisfy modulus. 
    But this follows from the containment lemma \cite[Proposition~2.4]{KP17} (see the proof of \cite[Lemmas~2.1 and 2.2]{KP12}) as $Z_l \subset W$, $Z_l \not\subset D$ and $f$ satisfies modulus $D$ on $W$. 

    We now have a well-defined pairing $\cap \colon \CH^p(X) \times \CH^q(X|D) \to \CH^{p+q}(X|D)$. The claim that it defines an action of $\CH^*(X)$ on $\CH^*(X|D)$ follows from the associativity of the intersection of cycles (if they intersect properly). This completes the proof of the theorem.

\section{Arithmetic Chow group with modulus} \label{sec*:ACH-M}

In this section, we attach a Green current $g_{Z}$ to a cycle $Z$ with modulus such that $dd^{c}g_{Z}+\delta_{Z}$ vanishes on $D$. Consequently, we define an arithmetic Chow group with modulus.

Let $k = (k, \Sigma, F_{\infty})$ be an arithmetic field and let $f \colon X \to \Spec(k)$ be a smooth quasi-projective $k$-scheme of dimension $d \geq 0$. We let $D \hookrightarrow X$ be an effective Cartier divisor such that $D_{{\rm red}}=\Sigma_{i}D_{i}$ is a simple normal divisor on $X$. We let $D_{\infty}\coloneqq D_{{\rm red}, \Sigma}(\C)$, and other notations be as in Sections~\ref{sec:Rec-ACG} and \ref{sec:CH-M}. 
Recall from Section~\ref{sec:CH-M} that 
 $\mathcal{Z}^{p}(X|D)$ denote the free abelian group generated by those codimension $p$ irreducible closed subsets of $X$ that do not meet $D$. Recall from \eqref{eqn:R-F-int-Forms*-1} that $A^{p,p}(X_{\R})$ is the space of real forms $\omega$ satisfying $F_{\infty}^{*}(\omega) = (-1)^p \omega$ and from Lemma \ref{lem:Surj*-2}, $\Sigma_{N}A^{\ast}_{M}$ is the space of smooth forms on a complex manifold $M$ that vanishes on a simple normal crossing divisor $N$. Let
\begin{equation} \label{eqn:real-form-M}
A^{p,p}(X_{\R}|D)  := \{\omega \in A^{p,p}(X_{\R})~|~\omega\in \Sigma_{D_{\infty}}A^{p,p}_{X_{\infty}}\}.
\end{equation}
We summarize the above condition by $\omega_{|D} = 0$ and call such forms as smooth real forms on $X$ that vanish along $D$.

Much like the usual algebraic cycles we show that for a cycle with modulus, there always exists a Green current with the vanishing condition along $D$ for the corresponding smooth form.

\begin{thm} \label{thm:Surj}
    Assume that $X$ is a smooth quasi-projective $k$-scheme. Then for a cycle $Z\in \caZ^{p}(X|D)$, there always exists Green current $g_{Z}$ such that $dd^{c}g_{Z}+\delta_{Z}\in A^{p,p}(X_{\R}|D)$.
\end{thm}
\begin{proof}
    Let $Z\in \mathcal{Z}^p(X|D)$ be a cycle with modulus of codimension $p$. We let $\textnormal{cl}(Z)$ be the class of $Z_{\infty}$ in $H^{p,p}_{Z_{\infty}}(X_{\infty})$. 
    Indeed, if $\fg_Z \in A^{p-1,p-1}(X_{\infty} \setminus Z_{\infty})$ is a Green form for $Z$ of logarithmic type (see \cite[1.3.5]{GS90} for the definition of a Green form of logarithmic type), and if $\omega_0 \in A^{p,p}(X_{\infty})$ is the smooth form such that $d d^c([\fg_Z]) + \delta_{Z} = [\omega_0]$, then the element $(\omega_0, \frac{-\iota}{2\pi} {\partial}\fg_Z) \in  A^{p,p}(X_{\infty})\oplus  A^{p,p-1}(X_{\infty} \setminus Z_{\infty})$
    maps to $\textnormal{cl}(Z) \in H^{p,p}_{Z_{\infty}}(X_{\infty})$ (see \cite[page~114 and Lemma~2.1.1 (i)]{GS90}).
    Since $Z\cap D= \emptyset$, we have a well-defined map 
    \[
    H^{p,p}_{Z_{\infty}}(X_{\infty}) = H^{p}(s(A^{p, *}(X_{\infty}) \to  A^{p,*}(X_{\infty} \setminus Z_{\infty}) )) \to 
    H^p(K^{p,*}_{X_{\infty}, D_{\infty}}).
    \] 
    We let $\textnormal{cl}_{X|D}(Z) \in H^p(K^{p,*}_{X_{\infty}, D_{\infty}})$ denote the image of $\textnormal{cl}(Z)$ under the above map. By Lemma \ref{lem:Surj*-2}, we get 
    \[
    {\cl}_{X|D}(Z) \in  H^{p}(K^{p,*}_{X_{\infty}, D_{\infty}})= 
    H^{p}(\Sigma_{D_{\infty}} A^{p,*}_{X_{\infty}})
    \]
    such that it maps to the cycle class $\cl_X(Z) \in 
    H^{p}(X_{\infty}, A^{p,*}_{X_{\infty}}) = H^{p,p}(X_{\infty})$ under the map induced by the inclusion $\Sigma_{D_{\infty}} A^{p,*}_{X_\infty} \to A^{p,*}(X_{\infty})$. 
    Let $\omega'_{Z|D} \in \Sigma_{D_{\infty}} A^{p,p}_{X_{\infty}}$ be such that its class is the element  ${\cl}_{X|D}(Z)$. We therefore get an element $\omega'_{Z|D} \in  A^{p,p}(X_{\infty})$ such that it maps to the class $\cl_X(Z)$
    in $H^{p,p}(X_{\infty})$ and $(\omega'_{Z|D})_{|D} = 0$.  

    Since $X_{\infty}$ is a complex manifold, the $d$, $\partial$ and $\overline{\partial}$ cohomology of currents on $X$ is the same as the corresponding cohomology of $C^{\infty}$ forms on $X$
    (for example, see \cite[Pages 382-385]{GH78}). Under this identification, the current $\delta_{Z}$ maps to the class $\cl_X(Z)$
    in $H^{p,p}(X_{\infty})$. In particular, the current $[\omega'_{Z|D}]$ and $\delta_{Z}$ are in same class in $H^{p,p}(X_{\infty})$ and hence there exists $g'_{Z} \in \wt{D}^{p-1,p-1}(X_{\infty}) $ such that 
    $\delta_{Z} + d d^c g'_{Z} = [\omega'_{Z|D}] \in {D}^{p,p}(X_{\infty}) $. 

    Let $g^{1}_{Z} = \frac{1}{2}(\tau(g'_{Z}) + g'_{Z})$, $g_{Z} = \frac{1}{2}(g^{1}_{Z} + (-1)^{p-1} F_{\infty}^*g^{1}_{Z}))$, and $\omega^{1}_{Z|D} = \frac{1}{2}(\omega'_{Z|D} + \tau(\omega'_{Z|D}))$, $\omega_{Z|D} = \frac{1}{2}(\omega^{1}_{Z|D} + (-1)^{p} F_{\infty}^*\omega^{1}_{Z|D}))$. We then have 
    $\omega_{Z|D} \in A^{p,p}(X_{\R}|D)$ and 
    $dd^c g_{Z}+ \delta_{Z}=[\omega_{Z|D}]\in D^{p,p}(X_{\R})$. This concludes the proof of the theorem. 
\end{proof}
Now we can define
\begin{df} \label{def:A-cycles*-M}
We let 
\begin{displaymath}\label{eqn:A-cycle*-M-1-1}
\wh{\mathcal{Z}}^p(X|D) = \{(Z, \wt{g}) \in  \mathcal{Z}^p(X|D) \oplus \wt{D}^{p-1,p-1}(X_{\R})~|~ \delta_Z + d d^c \wt{g} = [\omega], \textnormal{ where $\omega \in A^{p, p}(X_{\R}| D)$} \}.
\end{displaymath}
\end{df}

For a codimension $(p-1)$ irreducible closed subset $W$ such that $W\not\subset D$, we let $\nu: W^N \to X$ denote the projection from the normalization of 
$W$ and let $E = \nu^*(D)$ denote the pull-back of $D$ to $W^N$. Recall that $G(W^N, E)$ denotes the set of non-zero rational functions on $W^N$ that are regular (actually a unit) in a neighborhood of $E$ and one along $E$ (see \eqref{eqn:CH-M*-1-0}). For $f \in G(W^N, E)$,
we have 
\begin{equation} \label{eqn:A-cycle*-M-1-3}
  [-\log | f |^2] \in  {D}^{0,0}(\wt{W}_{\R}), 
 \end{equation}
where $\wt{W} \to W^N$ is a resolution of singularities such that the pull-back $\wt{E}$ of $E$ is supported on a  simple normal crossing divisor on $\wt{W}$ (see Example \ref{ex:real-currents}(ii)). 
Let $\iota_W \colon \wt{W}  \to X$ denote the induced proper morphism between smooth quasi-projective schemes over $k$. We have a well-defined pushforward map 
$\iota_{W*}\colon {D}^{0,0}(\wt{W}_{\R}) \to {D}^{p-1,p-1}(X_{\R}).$
We let 
\begin{equation}\label{eqn:A-cycle*-M-1-3.5}
\wh{{\rm div}(f)} = ({\rm div(f)}, \iota_{W*} [-\log | f |^2]) \in \mathcal{Z}^p(X|D) \oplus \wt{D}^{p-1,p-1}(X_{\R}).
\end{equation}
Since $\delta_{{\rm div(f)}} + d d^c (\iota_{W*}[-\log | f |^2]) = 0$,  
 we have $ \wh{{\rm div}(f)}  \in \wh{\mathcal{Z}}^p(X|D)$. Note that this element does not depend on the choice of the resolution of singularities $\wt{W}$ as it is determined by $W_{{\rm reg}}$ (see \cite[Page~127]{GS90}). 
We now let $\wh{\mathcal{R}}^p(X|D)$ be the subgroup of $\wh{\caZ}^{p}(X|D)$, generated by all such $\wh{\Div(f)}$ and define an \textit{Arithmetic Chow group with modulus} as the quotient:
\begin{equation} \label{eqn:A-cycle*-M-1-6}
\wh{\CH}^p(X|D) := \wh{\mathcal{Z}}^p(X|D)/ \wh{\mathcal{R}}^p(X|D). 
\end{equation}

Forgetting the Green currents, one can define a map 
    $\wh{\mathcal{Z}}^p(X|D) \xrightarrow{\zeta_{X|D}} \mathcal{Z}^p(X|D)$. The above map takes $\wh{\Div(f)}$ to $\Div(f)$. We therefore have the induced map 
    $\wh{\CH}^{p}(X|D) \xrightarrow{\zeta_{X|D}} \CH^p(X|D)$.
From Theorem \ref{thm:Surj} the corollary below is immediate. 

\begin{cor}\label{cor:Surj}
    Assume that $X$ is a smooth quasi-projective $k$-scheme. Then for all $p\geq 0$, the canonical map 
    $\zeta_{X|D} \colon \wh{\CH}^p(X|D) \surj {\CH}^p(X|D)$ is surjective. 
\end{cor}
\subsection{Functoriality}\label{SS2-5}
Much like the usual arithmetic Chow groups, we study the behaviour of their modulus counterpart for proper and flat smooth morphisms.
\begin{prop}\label{functoriality}
Let $\phi \colon (Y,E)\rightarrow (X,D)$ be a smooth morphism of modulus pairs (Definition \ref{df:modpair}). Then 
\begin{enumerate}
\item
 there is a pull-back homomorphism
\begin{displaymath}
\phi^{\ast}\colon \widehat{\CH}^{p}(X|D)\rightarrow  \widehat{\CH}^{p}(Y|E).
\end{displaymath}
\item
If $\phi$ is also proper, there is a push-forward homomorphism
\begin{displaymath}
\phi_{\ast}\colon \widehat{\CH}^{p}(Y|E)\rightarrow \widehat{\CH}^{p-e}(X|D),
\end{displaymath}
where $e=\dim(Y)-\dim(X)$.
\end{enumerate}
\end{prop}
\begin{proof}
The algebraic component is the proof of Proposition \ref{prop:functmodulus}, and we refer to \cite[Lemma 2.8, 2.9]{KP17} for the proof. We show that the corresponding analytic conditions are satisfied. For the pull-back morphism, we have a straightforward argument: Let $(Z,{g}_{Z})\in\wh{\CH}^{p}(X|D)$, with 
$\delta_Z + d d^c g_Z = [\omega]$ such that $\omega_{|_{D}}=0$. Then by the proof of \cite[Theorem~3.6.1]{GS90}, we have 
\begin{equation}
    \delta_{\phi^*(Z)} + d d^c (\phi^*g_Z) = [\phi^*(\omega)].
\end{equation}
But then $\phi^*(\omega)_{|_E} = (\phi_{|_E})^*(\omega_{|_{D}})=0$, where 
$\phi_{|_E} \colon \wt{E} \to \wt{D}$ is the induced map. 

For the push-forward, let $(W,{g}_{W})\in \wh{\CH}^{p}(Y|E)$, with  with 
$\delta_W + d d^c g_W = [\omega_W]$ such that $(\omega_W)_{|_{E}}=0$. Notice that, since $\phi$ is smooth, the element $\phi_{\ast}\omega_{W}$ can be uniquely defined satisfying $\phi_{\ast}[\omega_{W}]=[\phi_{\ast}\omega_{W}]$. We consider the following commutative diagram of complex manifolds. 
\begin{displaymath}
 \xymatrix@C4pc{
  Y_{\infty} \ar[r]^-{\phi} & X_{\infty} \\
  \wt{E}_{\infty} \ar[r]^-{\phi_{|_E}}\ar[u]^-{\iota_{E}} & \wt{D}_{\infty} \ar[u]^-{\iota_{D}},
  }
\end{displaymath}
where the vertical arrows are induced by the inclusions. 
 Using the projection formula for integration along fibers, we have 
 \[
 (\phi|_{E})_{\ast}(\iota^{\ast}_{E}\omega_{W})=\iota^{\ast}_{D}(\phi_{\ast}\omega_{W}). 
 \] 
 Since $\iota^{\ast}_{E}\omega_{W}=0$, we conclude that $\iota_{\ast}\omega_{W}$ vanish along $D$. This gives us the desired maps and completes the proof.
\end{proof}

\subsection{An exact sequence} \label{sec:es}

Recall from \cite[Theorem~3.3.5]{GS90} that arithmetic Chow groups fit into the following exact sequence. 
\[
\CH^{p, p-1}(X) \xrightarrow{\rho}  \wt{A}^{p-1, p-1}(X_{\R}) \xrightarrow{a} \wh{\CH}^{p}(X) \xrightarrow{\zeta} \CH^p(X) \to 0.  
\]
 We prove a similar exact sequence for the arithmetic Chow groups with modulus. Before that, we fix certain notations. 
 We let 
 \begin{equation} \label{eqn:es*-1}
     \wt{A}^{p-1, p-1}(X_{\R}, dd^{c}_{|D}=0) = \{\omega\in \wt{A}^{p-1, p-1}(X_{\R})| 
     (d d^c\omega)_{|D} = 0\} \subset  \wt{A}^{p-1, p-1}(X_{\R}). 
 \end{equation} 
Note that $\wt{A}^{p-1, p-1}(X_{\R}, dd^{c}_{|D}=0)$ contains the subgroup $ \wt{A}^{p-1, p-1}(X_{\R}|D)$ and  any exact $(p-1, p-1)$-form. For $U= X \setminus D$ and $x \in U^{(p-1)}$, we let  $k(x|D)^{\times}$ denote the subgroup $G(W_x^N, E_x) \subset k(x)^{\times}$, where $W$ is the closure of $x$ in $X$, $W_x = \Spec(\mathcal{O}_{W,x})$ and $E_x$ denotes the pull-back of $D$ along the induced morphism $W_x^N \to X$. We let 
\begin{equation*}  \label{eqn:es*-2}
    \mathcal{Z}^{p, p-1}(X|D) = {\rm Ker}\left(\oplus_{x\in U^{(p-1)}} k(x|D)^{\times} \xrightarrow{\text{div}} \oplus_{x\in X^{(p)}} \mathbb{Z}\right),
\end{equation*}
\begin{equation*} \label{eqn:es*-3}
    \mathcal{R}^{p,p-1}(X|D) = {\rm Im}(\oplus_{x\in X^{(p-2)}} K_2(k(x)) \to \oplus_{x\in X^{(p-1)}} k(x)^{\times}) \cap \mathcal{Z}^{p, p-1}(X|D)
\end{equation*}
and let 
\begin{equation} \label{eqn:es*-4}
    \CH^{p, p-1}(X|D) = \frac{\mathcal{Z}^{p, p-1}(X|D)}{\mathcal{R}^{p, p-1}(X|D)}. 
\end{equation}
Since $  \mathcal{Z}^{p, p-1}(X|D) \subset   \mathcal{Z}^{p, p-1}(X) := {\rm Ker}(\oplus_{x\in X^{(p-1)}} k(x)^{\times} \to \oplus_{x\in X^{(p)}} \mathbb{Z}) $, we have the following  inclusion. 
\begin{equation} \label{eqn:es*-5}
    \CH^{p, p-1}(X|D) \inj  \CH^{p, p-1}(X). 
\end{equation}

\begin{rmk} \label{rem:warning-eqn-es-4}
    Recall that by the Gersten sequence, we know that $\CH^{p, p-1}(X)$ is isomorphic to the Zariski-cohomology group $H^{p-1}(X, \mathcal{K}_{p})$, where $\mathcal{K}_*$ are Zariski sheaves associated to the $K$-groups, and that these groups are isomorphic to Bloch's higher Chow group $\CH^{p}(X,1)$ (Landburg \cite[Lemma~E]{Lan91}). However, the groups $\CH^{p, p-1}(X|D)$ are 
    not going to be isomorphic to $\CH^p(X|D, 1)$. We note that by construction that the map $\CH^{p, p-1}(X|D) \inj \CH^p(X, 1)$ is injective but the natural map $\CH^{p}(X|D,1) \to \CH^{p}(X,1)$ is not injective in general. With this viewpoint, we expect $\CH^{p, p-1}(X|D)$  to be the image of the map $\CH^{p}(X|D,1) \to \CH^{p}(X,1)$. 
    Note that, in \cite{GKR20}, it is shown to be the case for zero-cycles, i.e., when $p=d+1$. 
\end{rmk}
We construct various maps in the modulus setting. 
\begin{enumerate}
    \item We have a surjective map 
    $\zeta_{X|D}\colon\wh{\CH}^{p}(X|D) \rightarrow\CH^p(X|D)$ by Corollary~\ref{cor:Surj}. 
    \item We construct a map $a_{X|D} \colon
    \wt{A}^{p-1, p-1}(X_{\R}, dd^{c}_{|D}=0) \to \wh{\CH}^{p}(X|D)$. For  $\omega \in \wt{A}^{p-1, p-1}(X_{\R}, dd^{c}_{|D}=0)$, we consider the element 
     $(0, [\omega]) \in \wh{\mathcal{Z}}^p(X)$, 
    where $[\omega]$ is the associated current.
    Since 
    $(dd^c \omega)_{|D}=0$, we have $(0, [\omega]) \in \wh{\mathcal{Z}}^p(X|D)$. We let 
      \begin{equation} \label{eqn:es*-6}
        a_{X|D}(\omega)= (0, [\omega]) \in \wh{\CH}^p(X|D). 
    \end{equation}
    \item Finally, we construct a map 
    $\rho_{X|D} \colon \CH^{p, p-1}(X|D) \to  \wt{A}^{p-1, p-1}(X_{\R}, dd^{c}|_D=0)$. Recall from the proof of \cite[Theorem~3.3.5]{GS90} that the map 
    $\rho\colon \CH^{p, p-1}(X) \to  \wt{A}^{p-1, p-1}(X_{\R})$ is defined so that 
    $\rho(\{f_{i}\}) = - \sum_{i} \log |f_{i}|^2$, where $f_i$ are rational function on an irreducible closed subset $W_i$ such that $\sum_i \Div(f_{i})=0$. Indeed, since $\sum_i \Div(f_{i})=0$ by definition, the current $- \sum_{i} [\log |f_{i}|^2]$ is given by a smooth form that we keep denoting by $-\sum_{i} \log |f_{i}|^2$. Note that by inclusions \eqref{eqn:es*-1} and \eqref{eqn:es*-5}, it therefore suffices to show that 
    if $\{f_{i}\} \in \mathcal{Z}^{p, p-1}(X|D) = {\rm Ker}(\oplus_{x\in U^{(p-1)}} k(x|D)^{\times} \to \oplus_{x\in X^{(p)}} \mathbb{Z})$, then 
    $-\sum_{i} \log |f_{i}|^2 \in \wt{A}^{p-1, p-1}(X_{\R}, dd^{c}_{|D}=0)$. But this follows because the natural map $ A^{p, p}(X_{\R}) \inj D^{p, p}(X_{\R})$ is an inclusion  and 
    $\sum_{i} d d^c [ \log |f_{i}|^2] =  \delta_{\sum_i \Div(f_i)} = 0$. 
\end{enumerate}

We now study various compositions. Note that  if $\omega \in \wt{A}^{p-1, p-1}(X_{\R}, dd^{c}|_D=0)$, then 
    \begin{equation} \label{eqn:es*-7}
    \zeta_{X|D} \circ a_{X|D}( \omega) = \zeta_{X|D} (0, \omega) = 0. 
    \end{equation} 
If $\{f_{i}\} \in \mathcal{Z}^{p, p-1}(X|D) = {\rm Ker}(\oplus_{x\in U^{(p-1)}} k(x|D)^{\times} \to \oplus_{x\in X^{(p)}} \mathbb{Z})$, then 
    \begin{equation}  \label{eqn:es*-7.5}
    a_{X|D} \circ \rho_{X|D} (\{f_i\}) = a_{X|D}(-\sum_{i} \log |f_{i}|^2) = (0, -[\sum_{i} \log |f_{i}|^2])
    \end{equation}
    \[ \hspace{5cm} = (\sum_i \Div(f_i), -\sum_{i} \log |f_{i}|^2]) = \sum_i\wh{\Div(f_i)} = 0 \in \wh{\CH}^p(X|D).\] Here the last equality holds because for each $f_i \in k(x_i|D)^{\times}$, we have  $\wh{\Div(f_i)} \in \wh{\mathcal{R}}^p(X|D)$. 

By \eqref{eqn:es*-7} and \eqref{eqn:es*-7.5}, we have the following complex of abelian groups. 
\begin{equation} \label{eqn:es*-8}
\CH^{p, p-1}(X|D) \xrightarrow{\rho_{X|D}}  \wt{A}^{p-1, p-1}(X_{\R}, dd^{c}|_{D}=0) \xrightarrow{a_{X|D}} \wh{\CH}^{p}(X|D) \xrightarrow{\zeta_{X|D}} \CH^p(X|D) \to 0.  
\end{equation}

\begin{thm}\label{thm:exact-seq}
    With the above notations, the complex \eqref{eqn:es*-8} is exact. 
\end{thm}
\begin{proof}
    By Corollary \ref{cor:Surj}, the map $\zeta_{X|D}$ is surjective. Since  we have shown that \eqref{eqn:es*-8} is already a complex, it suffices to show that 
    ${\rm Ker}(\zeta_{X|D}) \subset {\rm Im}(a_{X|D})$ and  ${\rm Ker}(a_{X|D}) \subset {\rm Im}(\rho_{X|D})$.

    We first show that ${\rm Ker}(\zeta_{X|D}) \subset {\rm Im}(a_{X|D})$. Let ${(\alpha, g_{\alpha})} \in \wh{\CH}^{p}(X|D)$ such that 
    $0 = \zeta_{X|D} ({(\alpha, g_{\alpha})}) = {\alpha} \in \CH^p(X|D)$. In other words, $\alpha = \sum_i \Div(f_i) \in \mathcal{Z}^p(X|D)$, where $f_i \in k(x_i|D)^{\times}$ is a rational function on $W_i = \ov{\{x_i\}}$ that satisfies modulus $D$. We have \[
    {(\alpha, g_{\alpha})} = {(\alpha, g_{\alpha})- \sum_i \wh{\Div(f_i)}} = {(0, g_{\alpha} + \sum_i [\log |f_i|^2])} \in \wh{\CH}^{p}(X|D).\]
    %By the definition of the map $a_{X|D}$, it suffices to show that     the current $g_{\alpha} + \sum_i [\log |f_i|^2])$ comes from a smooth form $\omega$ such that $(d d^c \omega)|_D =0$.
    We note that \cite[Lemma~1.2.2(3)]{GS90} yields a smooth form $\omega \in \wt{A}^{p-1, p-1}(X_{\R})$ such that $g_{\alpha} + \sum_i [\log |f_i|^2] = [\omega] \in \wt{D}^{p-1, p-1}(X_{\R})$. %It now remains to show that $(d d^c \omega)|_D =0$. To see this, 
    Since 
     $(\alpha, g_{\alpha}) \in \wh{\mathcal{Z}}^p(X|D)$, there exists a smooth form $\omega_1 \in A^{p, p}(X_{\R})$ such that 
    $ \delta_{\alpha} + d d^c(g_{\alpha}) = [\omega_1]$ and $(\omega_1)_{|D} = 0$. Moreover, $\delta_{\alpha} + \sum_i d d^c (-[\log |f_{i}|^2])$=0 implies that
    $d d^c ([\omega]) = d d^c (g_{\alpha} + \sum_i [\log |f_i|^2) = [\omega_1]$. 
    Recall that the natural map $ A^{p, p}(X_{\R}) \inj D^{p, p}(X_{\R})$ is injective and hence $(d d^c \omega)_{|D} = (\omega_1)_{|D} = 0$. This completes the proof because we have proven that  $\omega \in \wt{A}^{p-1, p-1}(X_{\R}, dd^{c}_{|D}=0)$ and 
    \[
    a_{X|D} (\omega) = {(0, [\omega])} = {(0, g_{\alpha} + \sum_i [\log |f_i|^2])} =  {(\alpha, g_{\alpha})}. 
    \]

    Finally, we prove that ${\rm Ker}(a_{X|D}) \subset {\rm Im}(\rho_{X|D})$. Let $\omega_0 \in  \wt{A}^{p-1, p-1}(X_{\R}, dd^c_{|D}=0)$ such that 
    $0 = a_{X|D}(\omega_0)= {(0, [\omega_0])} \in \wh{\CH}^{p}(X|D)$. In other words, there exists finitely many rational functions $f_i \in k(x_i|D)^{\times}$ on $W_i = \ov{\{x_i\}}$ that satisfies modulus $D$ such that $(0, [\omega_0]) = \sum_i \wh{\Div(f_i)} = (\sum_i \Div(f_i), \sum_i -[\log |f_i|^2]) \in \wh{\mathcal{Z}}^p(X|D)$. 
    Observe that by definition of the rational equivalence with modulus, we have that $x_i \in U= X \setminus D$. Summing up these observations, we get an element 
    $\{f_i\} \in \mathcal{Z}^{p, p-1}(X|D) = {\rm Ker}(\oplus_{x\in U^{(p-1)}} k(x|D)^{\times} \to \oplus_{x\in X^{(p)}} \mathbb{Z})$ such that 
    \[
    \rho_{X|D}(\{f_i\}) = - \sum_i \log |f_i|^2 = \omega_0. 
    \]
    Note that the last equality holds because $[\omega_0] = \sum_i -[\log |f_i|^2]$ and the natural map $ A^{p-1, p-1}(X_{\R}) \inj D^{p-1, p-1}(X_{\R})$ is an inclusion. 
    This completes the proof. 
\end{proof}

As a corollary, we get the following result. 
\begin{cor} \label{cor:exact-seqs}
    We have the following commutative diagram of exact sequences. 
    \begin{equation} \label{eqn:es-diag}
        \xymatrix@C2pc{
       \CH^{p, p-1}(X|D) \ar[r]^-{\rho_{X|D}}
       \ar@{^{(}->}[d] & \wt{A}^{p-1, p-1}(X_{\R}, dd^{c}_{|D}=0) \ar[r]^-{a_{X|D}} \ar@{^{(}->}[d]& \wh{\CH}^{p}(X|D) \ar[r]^-{\zeta_{X|D}} \ar[d]& \CH^p(X|D) \ar[d] \ar[r] & 0\\
       \CH^{p, p-1}(X) \ar[r]^-{\rho} & \wt{A}^{p-1, p-1}(X_{\R}) \ar[r]^-{a}  & \wh{\CH}^{p}(X) \ar[r]^-{\zeta} & \CH^p(X) \ar[r] & 0. 
        }
    \end{equation}
\end{cor}
\begin{ex}\label{ex:5.8}
We use the exact sequence \ref{eqn:es*-8} to compute an interesting example. Let $X=\P^2_{\Q}$, the projective two space over $\Q$ with homogeneous coordinate $(x_0:x_1:x_2)$, and consider the line $D=V(x_0)$. In this case, since $D$ is ample, it intersects any projective curve in $X$. Hence $\CH^{1}(X|D)=0$. Using the exact sequence \ref{eqn:es*-8}, we get
\begin{displaymath}
\wh \CH^{1}(X|D)\cong\frac{A^{0,0}(X_{\R}, dd^{c}|_{D}=0)}{\text{Image}(\rho_{X|D})}.
\end{displaymath}
We note that $\CH^{1,0}(X|D)\hookrightarrow \CH^{1}(X,1)$. In this case,  $\CH^{1}(X,1)$ is a finite extension of $\Q$, and for any element $\alpha\in \CH^{1,0}(X|D)\hookrightarrow \CH^{1}(X,1)$, we know that $\rho_{X|D}(\alpha)=\log|\alpha|$. But the group $A^{0,0}(X_{\R}, dd^{c}|_{D}=0)$ is much larger. A typical example of an element in this group would be the smooth function
\begin{displaymath}
f([x_0:x_1:x_2])=\lambda+\frac{|x_0|^2}{|x_0|^2+|x_1|^2+|x_2|^2},~\lambda\in \R^{\times},
\end{displaymath}
satisfying $f|_{D}\neq 0$ but $dd^{c}f|_{D}=0$. So modulo the image of $\CH^{1,0}(X|D)$, the ingredients of the arithmetic Chow group with modulus are analytic in nature in this case. This can be compared to the fact that for the arithmetic affine line $\A^1_{\Z}=\Spec(\Z[T])$, the arithmetic Chow group $\wh \CH^1(\A^1_{\Z})\cong A^{0,0}_{\log}(\A^1_{\R})$, the space of smooth real valued functions on $\A^1_{\C}$ that are $F_{\infty}$-invariant, and has logarithmic singularities at infinity. At a more basic level, one can also compare with the isomorphism $\wh \CH^1(\Spec(\Q))\cong \frac{\R}{\langle \log|\alpha|;~\alpha\in \Q^{\times}\rangle}$. In both of these cases, the usual Chow groups are trivial, so the arithmetic version is determined entirely by the analytic component.
\end{ex}

%One the other extreme, $ \CH^{2}(X|D)$ is non-zero (as it has a surjection to $\Z$) and given a zero cycle $Z \subset X$ such that $D\cap Z =\emptyset$ and a green current $g_Z$ of $Z$, we get an element $(Z, g_Z) \in \wh{\mathcal{Z}}^2(X|D)$. In other words, for zero cycles with modulus, there is no modulus condition on the green currents. 
\begin{ex}
Let us discuss another example in which the Chow group with modulus is nonzero. Let $X = \P^1_{\Q} \times \P^1_{\Q}$ with projective bi-coordinates $((x:y), (z:w))$. On $\P^1_{\Q}$ we call $0\coloneqq (0:1)$ and $\infty\coloneqq (1:0)$. Let $D =\P^1_{\Q} \times 0$. Notice that $D$ is not ample, and the cycle $Z =\P^1_\Q \times \infty\in {\mathcal{Z}}^1(X|D)$. A Green form of logarithmic type at $\infty$ is given by:
\begin{displaymath}
g_{\infty} =\log\Big(\frac{|z|^2+|w|^2}{|w|^2}\Big).
\end{displaymath}
Pulling this back to $X$ gives a Green form of logarithmic type for $Z$, which we call $g_Z$. The associated current $[g_Z]\in D^{0,0}(X_{\infty})$ is a Green current of $Z$, satisfying $d d^c [g_Z] + \delta_{Z} =[dd^{c}g_Z]$. It is easy to see that $dd^{c}g_Z \in A^{1, 1}(X_{\R}| D)$, implying $(Z,[g_Z]) \in \wh{\mathcal{Z}}^1(X|D)$. Since adding a smooth form to a Green current also gives a Green current, and doesn't change the cohomology class, one can easily modify $[g_Z]$ to obtain another Green current of the form $[g_Z+\omega]$ where $\omega\in A^{0,0}(X_{\R})$, and such that $dd^{c}(g_Z+\omega)|_{D}\neq 0$. For example, the pull back of $\frac{|x|^2}{|x^2|+|y|^2}$ to $X$ is an element of $A^{0,0}(X_{\R})$. Let's call it $\omega$. Then
\begin{displaymath}
(Z,[g_Z+ \omega])\in \wh{\mathcal{Z}}^1(X) \setminus \wh{\mathcal{Z}}^1(X|D),
\end{displaymath}
because $dd^{c}\omega|_{D}\neq 0$.
\end{ex}
From the above examples, it is clear that our modulus condition is sensitive to both the algebraic and the complex analytic aspects of arithmetic cycles.

\section{Action of Arithmetic Chow group} \label{sec:Prodcut-ACGM}

Following \cite{GS90}, if $X$ is a smooth quasi-projective variety over an arithmetic field $k$, then the direct sum of arithmetic Chow groups $\widehat{\CH}^*(X) = \oplus_p \widehat{\CH}^p(X)$ forms a ring that is called the arithmetic Chow ring. Moreover, the canonical surjection $\widehat{\CH}^*(X) \surj \CH^*(X)$ is a ring homomorphism. Note that we do not need to tensor with $\Q$ because we are working with varieties over an arithmetic field. 

With respect to modulus, Theorem~\ref{thm:Prod-CGM} yields a natural $\CH^*(X)$-action on $\CH^*(X|D)$. 
In this section, we lift this action to a natural action of 
the arithmetic Chow ring $\widehat{\CH}^*(X)$ on the arithmetic Chow groups with modulus $\widehat{\CH}^*(X|D)$ where, as before, $D \hookrightarrow X$ is an effective Cartier divisor that is supported on a simple normal crossing divisor on $X$. The main result is the following theorem:

\begin{thm} \label{thm:Prod-1}
Let $k = (k, \Sigma, F_{\infty})$ be an arithmetic field and let $f \colon X \to \Spec(k)$ be a smooth quasi-projective variety of dimension $d \geq 0$. We let $D \hookrightarrow X$ be an effective Cartier divisor such that $D_{{\rm red}}$ is a simple normal crossing divisor on $X$. We then have an action 
\begin{equation} \label{eqn:Prod-1*1}
  \cap\colon  \widehat{\CH}^p(X) \times \widehat{\CH}^q(X|D) \to \widehat{\CH}^{p+q}(X|D)
\end{equation}
such that
\begin{enumerate}
    \item the action agrees with the product structure on $\widehat{\CH}^*(X)$ when $D=\emptyset$.

    \item the group homomorphism $\widehat{\CH}^*(X|D) \to \CH^*(X|D)$ is $\widehat{\CH}^*(X)$-module homomorphism, where $\CH^*(X|D)$ is seen as an 
    $\widehat{\CH}^*(X)$-module via the ring homomorphism $\widehat{\CH}^*(X) \to \CH^*(X)$.
    \item
    For a smooth morphism $\phi$ as in Proposition \ref{functoriality} with $X$ and $Y$ smooth and projective, this action satisfies the projection formula.
\end{enumerate}
\end{thm}

Recall from Section~\ref{sec:Action-CR} that if two cycles $\alpha = \sum_l n_l [Y_l] \in \mathcal{Z}^{p}(X)$ and $\beta = \sum_t m_t [Z_t] \in \mathcal{Z}^{q}(X)$ intersect properly, we have a well-defined cycle 
    $\alpha \cap \beta = \sum_{l, t} n_l m_t [Y_l \cap Z_t] \in  \mathcal{Z}^{p+q}(X)$. Moreover, if $\beta \in \mathcal{Z}^{q}(X|D)$, then $\alpha \cap \beta \in  \mathcal{Z}^{p+q}(X|D)$ because 
    $Y_l \cap Z_t \cap D \subset Z_t \cap D = \emptyset$ for all $l, t$. We start with discussing the action on Green currents. 

\subsection{Action on Green currents} \label{sec:Prod-currents}

In this section, we recall the definition of the product to green currents and its properties from \cite[Sections~2, 4]{GS90}.

Let $(\alpha,  {g}_{\alpha}) \in \wh{\mathcal{Z}}^p(X)$ and $(\beta, {g}_{\beta}) \in \wh{\mathcal{Z}}^q(X|D)$, i.e, 
$(\alpha,  {g}_{\alpha}) \in  \mathcal{Z}^p(X) \oplus \wt{D}^{p-1,p-1}(X_{\R})$ 
and $(\beta, {g}_{\beta}) \in  \mathcal{Z}^q(X|D) \oplus \wt{D}^{q-1,q-1}(X_{\R})$
such that 
\begin{equation} \label{eqn:prod-current*1}
    \delta_{\alpha} + d d^c {g}_{\alpha} = [\omega_{\alpha}] \textnormal{ and } 
      \delta_{\beta} + d d^c {g}_{\beta} = [\omega_{\beta}], 
\end{equation}
where $\omega_{\alpha} \in A^{p,p}(X_{\R})$ and $\omega_{\beta} \in A^{q,q}(X_{\R}| D)$. It follows from \cite[Theorem~1.3.5, Lemma~1.2.4]{GS90} that there exists Green form $\mathfrak{g}_{\alpha}'$ (resp. $\mathfrak{g}_{\beta}'$) of logarithmic type for the cycle $\alpha$ (resp. $\beta$) such that $[\mathfrak{g}_{\alpha}'] = {g}_{\alpha}  \in \wt{D}^{p-1,p-1}(X_{\R})$ 
(resp. $[\mathfrak{g}_{\beta}'] = \wt{g}_{\beta} \in  \wt{D}^{q-1,q-1}(X_{\R})$), where $[\mathfrak{g}] \in \wt{D}^{i-1,i-1}(X_{\R})$  denote the class of the current associated to a Green form $\mathfrak{g}$  logarithmic type of degree $i-1$. 
Assume now that the cycles $\alpha$ and $\beta$ intersect properly. 
We define the star product of the classes ${g}_{\alpha}$ and ${g}_{\beta}$ of Green currents as follows. 
\begin{equation} \label{eqn:prod-current*2}
{g}_{\alpha} \ast {g}_{\beta} := {[\mathfrak{g}_{\alpha}'] \ast [\mathfrak{g}_{\beta}']} 
= [\mathfrak{g}_{\alpha}']  \wedge \delta_{\beta} + \omega_{\alpha} \wedge [\mathfrak{g}_{\beta}'] \in 
\wt{D}^{p+q-1, p+q-1}(X_{\R}). 
\end{equation}
As proven in \cite[\S~2]{GS90}, the above element does not depend on the choices of $\mathfrak{g}_{\alpha}'$ and $\mathfrak{g}_{\beta}'$, and hence the product \eqref{eqn:prod-current*2} is well defined and it is the class of a Green current for the cycle $\alpha \cap \beta$. 

\begin{lem} \label{lem:prod-currents-1}
    Let $(\alpha,  {g}_{\alpha}) \in \wh{\mathcal{Z}}^p(X)$ and $(\beta, {g}_{\beta}) \in \wh{\mathcal{Z}}^q(X|D)$. Assume that the cycles $\alpha$ and $\beta$ intersect properly. Then 
    ${g}_{\alpha} \ast {g}_{\beta} \in \wt{D}^{p+q-1,p+q-1}(X_{\R}|D)$. In particular, we have 
    $(\alpha \cap \beta,{g}_{\alpha} \ast {g}_{\beta} ) \in 
    \wh{\mathcal{Z}}^{p+q}(X|D)$. 
\end{lem}
\begin{proof}
    Since the cycles $\alpha$ and $\beta$ intersects properly, it follows from 
    \cite[Theorem~2.1.4]{GS90} that 
    \begin{equation} \label{eqn:prod-current*3}
        d d^c({g}_{\alpha} \ast {g}_{\beta}) + \delta_{\alpha \cap \beta} = [\omega_{\alpha}\wedge\omega_{\beta}] \in D^{p+q}(X_{\R}), 
    \end{equation}
    where $\omega_{\alpha}$ and $\omega_{\beta}$ are as in \eqref{eqn:prod-current*1}. 
    Since $\omega_{\beta} \in A^{q,q}(X_{\R}|D)$, we have $\omega_{\alpha}\wedge\omega_{\beta} \in A^{p+q,p+q}(X_{\R}|D)$, i.e., we have $(\alpha \cap \beta,{g}_{\alpha} \ast {g}_{\beta} ) \in 
    \wh{\mathcal{Z}}^{p+q}(X|D)$. This completes the proof. 
\end{proof}

 At the level of rational equivalences, we have the following lemma. 
\begin{lem}\label{lem:prod-curr-rel}
    \begin{enumerate}
    \item Let $(\alpha,  {g}_{\alpha}) \in \wh{\mathcal{R}}^p(X)$ and $(\beta, {g}_{\beta}) \in \wh{\mathcal{Z}}^q(X|D)$. Assume that the cycles $\alpha$ and $\beta$ intersect properly. Then 
        $(\alpha \cap \beta,{g}_{\alpha} \ast {g}_{\beta} ) \in 
    \wh{\mathcal{R}}^{p+q}(X|D)$.
  
    \item Let $(\alpha,  {g}_{\alpha}) \in \wh{\mathcal{Z}}^p(X)$ and $(\beta, {g}_{\beta})= \wh{\Div(f)} \in \wh{\mathcal{R}}^q(X|D)$, where $f$ is a rational function on a closed subvariety $W$ of $X$. Assume further that the cycle $\alpha$  intersects with $W$, $W \cap D$ and $\Div(f)$ properly. Then $(\alpha \cap \beta,{g}_{\alpha} \ast {g}_{\beta} ) \in 
    \wh{\mathcal{R}}^{p+q}(X|D)$.
 \end{enumerate}
\end{lem}
\begin{proof}
    The lemma follows from \cite[Lemmas~4.2.5 and 4.2.6]{GS90} and the proof of Theorem
    \ref{thm:Prod-CGM}. Below, we write the details of both parts. 
    
Let $(\alpha,  {g}_{\alpha}) \in \wh{\mathcal{R}}^p(X)$. By the moving lemma for $K_1$-chains \cite[Lemma~4.2.6]{GS90}, we can find finitely many closed subvarieties $W_l$ of $X$ and rational functions $h_l$ on $W_l$ such that 
        $(\alpha, {g}_{\alpha})= \sum_l m_l \wh{\Div(h_l)}$ and that each $\Div(h_l)$ intersects with $\beta$ properly. But then for the cycle 
        $\beta = \sum_t n_t [Z_t] \in {\mathcal{Z}}^q(X|D)$, we have 
        \begin{eqnarray*}
            {g}_{\alpha} \ast {g}_{\beta}  &=& - \sum_l m_l ([{\rm log} |h_l|^2] \ast {g}_{\beta})\\
            & =^1&  - \sum_l m_l ([{\rm log} |h_l|^2] \wedge \delta_{\beta})\\
            & =^2& - \sum_l m_l (\log |h_l  \cdot \beta|^2). 
        \end{eqnarray*}
        where the equality $=^1$ holds because of definition on the star-product of currents \eqref{eqn:prod-current*2} and the fact that $d d^c(-[\log |h_l|^2]) + \delta_{\Div(h_l)} = 0$. Note that the element
        $h_l  \cdot \beta = \sum_t n_t (h_l \cdot Z_t) $ is defined using \eqref{eqn:Prod-cycle*-7} and 
        the equality $=^2$ holds by \cite[Lemma~4.2.5(2)]{GS90}. For each $l$ and $t$, we have 
        $h_l  \cdot Z_t \in \mathcal{R}^*(X|D)$ because $h_l  \cdot Z_t$ is a rational function on a closed subvariety that is contained inside $W_l \cap Z_t$ and $Z_t \cap D = \emptyset$ (see the proof of Lemma \ref{lem:Prod-cycle-2}). Moreover, by \eqref{eqn:Prod-cycle*-8}, we have 
     $\alpha \cdot \beta = \sum_{l, t} m_l n_t \Div(h_l) \cdot [Z_t] = 
     \sum_{l, t} m_l n_t \Div(h_l\cdot Z_t) $. Putting all together, we get 
     \begin{equation} \label{eqn:prod-curr-rel*1}
         (\alpha \cap \beta,{g}_{\alpha} \ast {g}_{\beta} ) = \sum_{l,t} m_l n_t\wh{\Div(h_l  \cdot Z_t)}  \in 
    \wh{\mathcal{R}}^{p+q}(X|D). 
     \end{equation}
     This proves the part (i) of the lemma.

    Let $(\alpha, {g}_{\alpha}) \in \wh{\mathcal{Z}}^p(X)$ and $(\beta, {g}_{\beta})= \wh{\Div(f)} \in \wh{\mathcal{R}}^q(X|D)$ be as in part (ii). Let $\alpha=\sum_l m_l [V_l] \in \mathcal{Z}^p(X)$.
    Then since we have assumed that each $V_l$ intersects with $W$ properly, we have 
    \begin{eqnarray*}
            {g}_{\alpha} \ast {g}_{\beta}  &=& {g}_{\alpha} \ast - ([{\rm log} |f|^2]) \\
            & =^3& -  [{\rm log} |f|^2] \wedge \delta_{\alpha} \\
            & =^4& - \sum_l m_l [\log |f\cdot V_l|^2], 
    \end{eqnarray*}
    where the equality $=^3$ holds because ${g}_{\alpha} \ast ([{\rm log} |f|^2]) =  ([{\rm log} |f|^2]) \ast {g}_{\alpha} = [{\rm log} |f|^2] \wedge \delta_{\alpha}$ and as before, the equality $=^4$ holds by \cite[Lemma~4.2.5(2)]{GS90}.  Note that if $[V_l \cap W] = \sum_{t} n^l_t W^l_{t}$, then $f\cdot V_l= \sum_{t} n^l_t (f|_{W^l_t})$ (see \eqref{eqn:Prod-cycle*-7}). Moreover, as in the proof of Theorem 
    \ref{thm:Prod-CGM}, $W^l_t \not\subset D$ and 
    the rational functions $f|_{W^l_t}$ satisfy modulus by containment lemma. By \eqref{eqn:Prod-cycle*-10}, we have 
    $\alpha \cap \beta = \alpha \cap \Div(f) = \sum_l m_l ([V_l] \cap \Div(f)) = \sum_l m_l \Div(f \cdot V_l)$. In particular, we get 
    \begin{equation} \label{eqn:prod-curr-rel*2}
        (\alpha \cap \beta,{g}_{\alpha} \ast {g}_{\beta} ) =  
        \sum_l m_l \wh{\Div(f \cdot V_l)} \in \wh{\mathcal{R}}^{p+q}(X|D).
    \end{equation}
    This proves the part (ii) and hence the lemma. 
\end{proof}

In the following proof of Theorem~\ref{thm:Prod-1}, we use the same notations for the elements of $\wh{\mathcal{Z}}^p(X)$ (resp. $\wh{\mathcal{Z}}^q(X|D)$) and their class in 
$\wh{\CH}^p(X)$ (resp. $\wh{\CH}^q(X|D)$). 

\subsection{Proof of Theorem \ref{thm:Prod-1}}

 Let $(\beta, {g}_{\beta}) \in \wh{\mathcal{Z}}^q(X|D)$ and let ${(\alpha, {g}_{\alpha})} \in \wh{\CH}^p(X)$.  By Chow's moving lemma \cite[\S~3]{Roberts}, we can find a cycle $\alpha_0 \in \mathcal{Z}^p(X)$ such that $\alpha_0$ intersects with $\beta$ properly and it is rationally equivalent to $\alpha$. Let $\alpha_0 - \alpha = \sum_l \Div(f_l) \in \mathcal{R}^p(X)$, where $f_l$ are rational functions on closed subvarieties $W_l$ of $X$. If we let $g_0 = g_{\alpha} - \sum_l [\log |f_l|^2]$, then $g_0$ is a Green current for the cycle $\alpha_0$ and 
 \begin{equation} \label{eqn:thm-Prod-1*1}
     {(\alpha, {g}_{\alpha})} = {(\alpha_0, {g_0})} \in  \wh{\CH}^p(X). 
 \end{equation}
 Since $\alpha_0$ intersects with $\beta$ properly, Lemma \ref{lem:prod-currents-1} yields a well-defined element 
 $(\alpha_0 \cap \beta, g_0\ast g_{\beta}) \in \wh{\mathcal{Z}}^{p+q}(X|D)$. 
 We now define 
 \begin{equation}\label{eqn:thm-Prod-1*2}
     {(\alpha, {g}_{\alpha})} \cap (\beta, {g}_{\beta}):= 
 {(\alpha_0 \cap \beta, g_0 \ast g_{\beta})}  \in \wh{\CH}^{p+q}(X|D). 
 \end{equation}
Let $\alpha_0' \in \mathcal{Z}^p(X)$ be another cycle that intersects with $\beta$ properly and it is rationally equivalent to $\alpha$. Let $\alpha_0' - \alpha= \sum_t \Div(f_t')$ and let $g_0' = g_{\alpha} - \sum_t [\log |f_t'|^2]$ be the corresponding Green current for the cycle $\alpha_0'$. Then 
$(\alpha_0, {g_0}) - (\alpha_0', {g_0'}) = \sum_l \wh{\Div(f_l)} - \sum_t \wh{\Div(f_t')}  =: \sum_s \Div(h_s)$. But then 
\begin{eqnarray}\label{eqn:thm-Prod-1*3}
(\alpha_0 \cap \beta, g_0 \ast g_{\beta}) - (\alpha_0' \cap \beta, g_0' \ast g_{\beta}) &=& ((\alpha_0- \alpha_0') \cap \beta, (g_0 - g_0') \ast g_{\beta})\\
\nonumber & =& \sum_s (\Div(h_s) \cap \beta , -[\log |h_s|^2] \ast g_{\beta}). 
\end{eqnarray}
By Lemma \ref{lem:prod-curr-rel}(i), it follows that this difference lie in $\wh{\mathcal{R}}^{p+q}(X|D)$. The element 
\eqref{eqn:thm-Prod-1*2} is therefore well-defined and does not depend on the choice of $\alpha_0$. Since the intersection product of cycles and the product of currents are linear in both coordinates, we get a bilinear pairing 
    $\cap \colon \wh{\CH}^p(X) \times \wh{\mathcal{Z}}^q(X|D) \to \wh{\CH}^{p+q}(X|D)$.

Let $W$ be a closed subscheme of $X$ of codimension $q-1$ and let $f \in G(W^N, E)$, where $E$ is the pull-back of $D$ to $W^N$ (see \S~\ref{sec:CH-M} for the definition of $G(W^N, E)$). Let ${(\alpha, {g}_{\alpha})} \in \wh{\CH}^p(X)$. We need to show that
 ${(\alpha, {g}_{\alpha})} \cap \wh{\Div(f)} = 0 \in \wh{\CH}^{p+q}(X|D)$. By Chow's moving lemma, we can find a cycle $\alpha_0 \in \mathcal{Z}^p(X)$ such that $\alpha_0$ intersects with $W$, $W \cap D$  and $\Div(f)$ properly, and it is rationally equivalent to $\alpha$. Let $g_0$ be the associate Green current for $\alpha_0$ as in the previous paragraph. But then by \eqref{eqn:thm-Prod-1*2} and  Lemma \ref{lem:prod-curr-rel}(ii), we have
 \begin{equation}\label{eqn:thm-Prod-1*4}
      {(\alpha, {g}_{\alpha})} \cap \wh{\Div(f)}= 
 {(\alpha_0 \cap \Div(f), g_0 \ast (-[\log |f|^2]))} = 0 \in  \wh{\CH}^{p+q}(X|D). 
 \end{equation}
This completes the proof of the existence of the bilinear pairing \eqref{eqn:Prod-1*1}.

The rest of the properties of the theorem follow from the corresponding properties of Chow groups with modulus and of arithmetic Chow groups.
\qed

\section{Relative Hermitian Picard group} 
\label{sec:codim-one}

In this section, we generalize the results of \ref{sec:APG} and \ref{sec:Rel-Pic}, and define a Hermitian metric on the relative line bundles. Subsequently, we prove that the arithmetic Chow group $\wh{\CH}^1(X|D)$ is isomorphic to the Picard group of these relative line bundles with a Hermitian metric. We use the notations of Sections \ref{sec:APG} and \ref{sec:Rel-Pic}, and begin by defining the metric.

As before, let $f \colon X \to \Spec(k)$ be a smooth quasi-projective scheme of dimension $d \geq 0$ over an arithmetic field $k$ and let $D \hookrightarrow X$ be an effective Cartier divisor such that $D_{{\rm red}}$ is a simple normal divisor on $X$. 
Recall from Definition \ref{defn:coho-chern-class} that given a Hermitian line bundle $(\sL, ||\cdot||)$ on $X$, we have the cohomological  Chern class ${c}_1(\sL, ||\cdot||)$. Also, 
recall from Definition \ref{def:Rel-Pic} that for a line bundle $\sL$, a trivialization of $\sL$ along $D$ means an isomorphism $\phi\colon \sL_{|D} \xrightarrow{\cong} \sO_D$ that lifts to an open neighbourhood of $D$. We combine both these notions associated with a Hermitian line bundle in the following definition.

\begin{df} \label{def:A-Pic-R-W*-1}
The triple $(\sL, ||\cdot||, \phi)$ is said to be a Hermitian line bundle on $X$ with trivialization along $D$ if $(\sL, ||\cdot||)$ is a Hermitian line bundle on $X$, $\phi\colon \sL_{|D} \xrightarrow{\cong} \sO_D$ is a trivialization of $\sL$ along $D$ and ${c}_1(\sL, ||\cdot||)\in A^{1,1}(X_{\R}|D)$. 
\end{df}

We say two Hermitian line bundles
$(\sL_1, ||\cdot||_1, \phi_1)$ and $(\sL_2, ||\cdot||_2, \phi_2)$  with trivialization along $D$ are equivalent (under the relation $\sim$) if there exists an isomorphism $\Phi \colon \sL_1 \xrightarrow{\cong} \sL_2$ such that $\Phi^*(||\cdot||_2) = ||\cdot ||_1$ and there exists an neighborhood $U$ of $D$ where both $\phi_1$ and $\phi_2$ are defined and $\phi_2 \circ \Phi_{|U} = \phi_1 \colon (\sL_1)_{|U} \to \sO_U$, i.e., the following diagram commutes. 
 \begin{equation} \label{eqn:triv-1}
  \xymatrix@C.8pc{
  (\sL_1)_{|U} \ar[r]^-{\Phi_{|U}}_{\cong} \ar[dr]_-{\phi_1} & (\sL_2)_{|U} \ar[d]^-{\phi_2}\\
  &\sO_U.}
 \end{equation}
It is easy to see that the above relation is an equivalence relation. We shall denote the equivalence class of a Hermitian line bundle $(\sL, ||\cdot||, \phi)$  with trivialization along $D$ 
by the same symbol $(\sL, ||\cdot||, \phi)$. 
In the above case, we shall call $\Phi$ to be the isomorphism between  triples $(\sL_1, ||\cdot||_1, \phi_1)$ and $(\sL_2, ||\cdot||_2, \phi_2)$. 

\begin{df} \label{def:A-Pic-R-W*}
We define the relative Hermitian Picard group $\wh{\Pic}(X, D)$ as the set of the  equivalence classes of Hermitian line bundles on $X$ with trivialization along $D$, i.e., 
\begin{equation} \label{eqn:A-Pic-R-W*-2}
\wh{\Pic}(X, D):=\{(\sL, ||\cdot||, \phi)~|~{c}_1(\sL, ||\cdot||)\in A^{1,1}(X_{\R}|D)\}/ \sim,
\end{equation}
where the equivalence relation $\sim$ is as defined above. 
\end{df}
 
\begin{lem} \label{lem:A-Pic-R-W-Group}
There exists a natural abelian group structure on $\wh{\Pic}(X, D)$ so that 
\begin{equation} \label{eqn:A-Pic-R-W-Gp*-1}
(\sL_1, ||\cdot||_1, \phi_1)\cdot (\sL_2, ||\cdot||_2, \phi_2) =
(\sL_1 \otimes \sL_2,||\cdot||_1 \otimes ||\cdot||_2, \phi_1 \otimes  \phi_2).
\end{equation}
\end{lem}
 \begin{proof}
    By Definition \ref{def:A-Pic-R-W*}, it follows that the above product is well-defined, it is associative, and $(\sO_X, |\cdot|, \id)$ is the identity of the product. The commutativity of the product is also immediate from the definition of the equivalence relation. We show that every element $(\sL, || \cdot||, \phi) \in \wh{\Pic}(X, D)$ has an inverse. 

Let $U \subset X$ be an open neighborhood of $D$ such that 
there exists an isomorphism ${\phi'} \colon \sL_{|U} \to \sO_U$ satisfying  $\phi = {{\phi}'}_{|D}$. 
We let $(\phi')^{\vee} \colon \sO_U \xrightarrow{\cong}  \sL_{|U}^{-1}$ is the dual isomorphism obtained by tensoring ${\phi'}$ with $\sL_{|U}^{-1}$. We now denote by $1/\phi \colon \sL^{-1}_{|D} \xrightarrow{\cong} \sO_D$ the restriction of the isomorphism 
$((\phi')^{\vee})^{-1} \colon  \sL_{|U}^{-1} \xrightarrow{\cong} \sO_U$. 
Recall that given line bundle $\sL$ with a Hermitian form $||\cdot||$
that is locally given by smooth function $h_{\alpha, \infty} \colon U_{\alpha,\infty} \to \R_{>0}$, there exists a Hermitian form $||\cdot||^{-1}$ on $\sL^{-1}$ such that it is locally given by smooth function $h_{\alpha, \infty}^{-1} \colon U_{\alpha, \infty} \to \R_{>0}$. In particular, 
$c_1 (\sL^{-1}, ||\cdot||^{-1}) = - c_1(\sL, ||\cdot||)$. 
It therefore follows that 
$(\sL^{-1}, ||\cdot||^{-1}, 1/\phi) \in \wh{\Pic}(X,D)$ and it is the inverse of the element 
$(\sL, ||\cdot||, \phi)$. 
This completes the proof.  
\end{proof}

We now define the arithmetic Chern class of a Hermitian line bundle with trivialization.
Let $(\sL, ||\cdot||, \phi)$ be a  Hermitian line bundle with trivialization along $D$. Let $\phi'\colon \sL_{|U} \xrightarrow{\cong} \sO_U$ be a lift of $\phi$ for an open neighborhood $D \subset U$, i.e., $\phi'_{|D} = \phi \colon \sL_{|D} \xrightarrow{\cong} \sO_D$. We then define the arithmetic Chern class of $(\sL, ||\cdot||, \phi)$ as follows. 
\begin{equation} \label{eqn:Chern-HLB-W*-0}
\wh{c}_1(\sL, ||\cdot||, \phi) =
 (\Div(\phi'^{-1}), -[\log ||(\phi')^{-1} ||^2]) \in \mathcal{Z}^{1}(X|D) \oplus {D}^{0,0}(X_{\infty}).
\end{equation}
Since $\phi$ is a trivialization, by Definition \ref{def:A-Pic-R-W*-1}, we have ${c}_1(\sL, ||\cdot||)_{|D}=0 \in A^{1,1}(\wt{D}_{\infty})$. Since 
$\delta_{\Div(\phi'^{-1})} + d d^c (-[\log ||(\phi')^{-1} ||^2] ) = [{c}_1(\sL, ||\cdot||)]$, it follows from Definition \ref{def:A-cycles*-M} that $\wh{c}_1(\sL, ||\cdot||, \phi) \in \wh{\mathcal{Z}}^1(X|D)$. We let its image in $\wh{\CH}^1(X|D)$ to be denoted by the same symbol. 
As an element of $\wh{\CH}^1(X|D)$, the Chern class $\wh{c}_1(\sL, ||\cdot||, \phi)$ does not depend on the chosen extension $\phi'$ and it is invariant for the equivalence relation on Hermitian line bundles with trivialization. In short, the following definition now makes sense.

\begin{df}\label{def:Chern-HLB-W}
Given $(\sL, ||\cdot||, \phi) \in \wh{\Pic}(X, D)$, we define its arithmetic Chern class as follows. 
\begin{equation} \label{eqn:Chern-HLB-W*-2}
\wh{c}_1(\sL, ||\cdot||, \phi) =
 (\Div(\phi'^{-1}), -[\log ||\phi'^{-1} ||^2]) \in \wh{\CH}^1(X|D),
\end{equation}
where $\phi'$ is a lift of the isomorphism $\phi$. 
\end{df}

\begin{lem}\label{lem:Chern-HLB-WN-Hom}
The arithmetic Chern class defines a group homomorphism.  
\begin{equation} \label{eqn:Chern-HLB-WN-Hom*-1}
\wh{c}_1 \colon \wh{\Pic}(X,D) \to \wh{\CH}^1(X|D). 
\end{equation}
\end{lem}
\begin{proof}
 Let  $(\sL_1, || \cdot ||_1, \phi_1)$ and $ (\sL_2, || \cdot ||_2, \phi_2) \in \wh{\Pic}(X, D)$ be two elements and for $i=1,2$, let the isomorphism $\phi_i'\colon (\sL_i)_{|U} \xrightarrow{\cong} \sO_U$ be the lift of $\phi_i$.   
 We then have 
 \begin{eqnarray}\label{eqn:Chern-HLB-WN-Hom*-3}
 \nonumber 
 \wh{c}_1((\sL_1, || \cdot ||_1, \phi_1) \otimes (\sL_2, || \cdot ||_2, \phi_2)) & = & \wh{c}_1(\sL_1\otimes \sL_2, || \cdot ||_1\otimes ||\cdot||_2, \phi_1\otimes \phi_2)\\
 \nonumber & =&  (\Div(\phi_1'^{-1} \otimes \phi_2'^{-1}), -[\log (||\phi_1'^{-1} ||_1^2\ ||\phi_2'^{-1}||_2^2 )])\\
\nonumber  & = & (\Div(\phi_1'^{-1} ), -[\log (||\phi_1'^{-1} ||_1^2]) \\
\nonumber & &\hspace{2cm} + 
 ( \Div(\phi_2'^{-1}), -[\log ( ||\phi_2'^{-1}||_2^2 )])\\
 & = & \wh{c}_1(\sL_1, || \cdot ||_1, \phi_1) + \wh{c}_1(\sL_2, || \cdot ||_2, \phi_2). 
 \end{eqnarray}
This completes the proof. 
\end{proof}

We now prove the main result of this section. 

\begin{thm}\label{thm:Chern-Iso}
The arithmetic Chern class group homomorphism \eqref{eqn:Chern-HLB-WN-Hom*-1}
fits into the following commutative diagram such that the vertical arrows are isomorphisms. 
\begin{equation} \label{eqn:Chern-HLB-WN-Hom*-1.5}
    \xymatrix@C.8pc{
    \wh{\Pic}(X,D) \ar[r] \ar[d]^-{\wh{c}_1}_-{\cong}& \wh{\Pic}(X) \oplus {\Pic}(X,D) 
    \ar[d]^-{(\wh{c}_1, c_1)}_-{\cong}\\
    \wh{\CH}^1(X|D) \ar[r] & \wh{\CH}^1(X) \oplus {\CH}^1(X|D),
    }
\end{equation}
where the right vertical arrow is given by the isomorphisms \eqref{eqn:A-Chern-cl-3.1} and \eqref{eqn:c1-M*-2}. 
    \end{thm}
\begin{proof}
    The commutativity of \eqref{eqn:Chern-HLB-WN-Hom*-1.5} follows from the definition of the vertical arrows. By Propositions \ref{prop:A-Chern-class-3} and \ref{prop:c1-M-Iso}, it remains to show that the left vertical arrow is an isomorphism. As in the referred lemmas, we explicitly construct the inverse $\cyc_{X|D} \colon \wh{\CH}^1(X|D) \to  \wh{\Pic}(X, D)$ of the left vertical arrow as follows. 

    Let $(Z, g) \in \wh{\mathcal{Z}}^1(X|D) \subset  \mathcal{Z}^1(X|D) \oplus {D}^{0,0}(X_{\R})$. Since $X$ is a smooth variety, the Weil divisors on $X$ are the same as the Cartier divisors on $X$. As in the proofs of Propositions \ref{prop:A-Chern-class-3} and \ref{prop:c1-M-Iso}, we let $(U_{\alpha}, f_{\alpha})$ be a Cartier divisor for $Z$. By the proof of 
    Proposition~\ref{prop:A-Chern-class-3}, we have 
    a Hermitian line bundle 
\begin{equation} \label{eqn:Chern-Iso-*0}
(\sO_X(U_{\alpha}, f_{\alpha}), || \cdot||_g) = (U_{\alpha}, g_{\alpha \beta}= f_{\alpha}/f_{\beta}, h_{\alpha, \sigma} = {{\rm exp}}^{-(g)|_{U_{\alpha, \sigma}} - [\log |\sigma^*(f_{\alpha})|^2]}).
\end{equation}
Moreover, Proposition~\ref{prop:c1-M-Iso} yields a  trivialization $\phi \colon \sO_X(U_{\alpha}, f_{\alpha}))_{|D} \xrightarrow{\cong} \sO_{D}$ of the line bundle  $\sO(U_{\alpha}, f_{\alpha})$. Indeed, if  we let $U = X \setminus |Z|$, then $f_{\alpha} \in \sO(U \cap U_{\alpha})^{\times}$ and they patch up to yield a section $s \in \Gamma(U, \sO(U_{\alpha}, f_{\alpha}))$ such that $s \colon \sO_{U}\xrightarrow{\cong} \sO(U_{\alpha}, f_{\alpha}))_{|U}$. We let $\phi=(s_{|D})^{-1} \colon \sO(U_{\alpha}, f_{\alpha}))_{|D} \xrightarrow{\cong} \sO_{D}$ be the induced trivialization of the line bundle  $\sO(U_{\alpha}, f_{\alpha})$.
We therefore get the Hermitian line bundle $( \sO(U_{\alpha}, f_{\alpha}), || \cdot ||_g, \phi=(s_{|D})^{-1} )$ on $X$ with trivialization along $D$. 
We let its class in $\in \wh{\Pic}(X, D)$ be denoted by $\cyc_{X|D}(Z,g)$. In other words, we let  
\begin{equation} \label{eqn:ACM-M*11}
\cyc_{X|D}(Z,g) :=  ( \sO(U_{\alpha}, f_{\alpha}), || \cdot ||_g, \phi=(s_{|D})^{-1} ) \in \wh{\Pic}(X, D).
\end{equation}
One needs to check that the above element does not depend on the chosen line bundle from the isomorphic class of line bundles corresponding to the transition functions $g_{\alpha \beta}$, and it also does not depend on the Cartier divisor for $Z$, i.e., does not depend on the transition functions $f_{\alpha}$. These checks are routine and we leave them to the reader. For $f \in G(X,D)$, we have $\cyc_{X|D} (\wh{\Div(f)}) = (\sO_X, | \cdot |, 1) = 0 \in   \wh{\Pic}(X, D)$. In particular, we have a group homomorphism $\cyc_{X|D} \colon \wh{\CH}^1(X|D) \to \wh{\Pic}(X, D)$. We claim that this is an inverse of the arithmetic Chern class group homomorphism \eqref{eqn:Chern-HLB-WN-Hom*-1}. The proof is similar to the proofs of Propositions \ref{prop:A-Chern-class-3} and \ref{prop:c1-M-Iso}. We include it here for the sake of completion. 

We first compute the composition $\cyc_{X|D} \circ \wh{c}_1$.   Let $(\sL, || \cdot ||, \phi) \in  \wh{\Pic}(X, D)$. Let 
$\{U_{\alpha}\}$ be a cover of $X$ such that $\sL$ is trivial along $U_{\alpha}$ and let $U$ be an open neighbourhood of $D$ with a lift $\phi' \colon \sL_{|U} \xrightarrow{\cong} \sO_U$ of the trivialisation  $\phi$. Then the rational section $(\phi')^{-1} \colon \sO_U \xrightarrow{\cong}\sL_{|U} $ of $\sL$ is given by rational functions $f_{\alpha} \in \sO(U \cap U_{\alpha})$.  In particular, the Weil divisor 
$ \Div((\phi')^{-1})$ corresponds the Cartier divisor $(U_{\alpha} , f_{\alpha})$. 
If the Hermitian metric $||\cdot||$ is locally defined by the functions $h_{\alpha, \sigma}$ (on $U_{\alpha, \sigma}$), then the associated current $g = -[\log || (\phi')^{-1} ||^2]$ has the restrictions 
$g_{|U_{\alpha, \sigma}} = -[\log (f_{\alpha, \sigma} \ov{f_{\alpha, \sigma}} h_{\alpha, \sigma})] $. We therefore have
\begin{eqnarray*}
\cyc_{X|D} \circ \wh{c}_1 (\sL, || \cdot ||, \phi) &=& \cyc_{X|D} ( \Div((\phi')^{-1}), -[\log || (\phi')^{-1} ||^2])\\
& =^1& (\sO(U_{\alpha}, f_{\alpha}), {{\rm exp}}^{-(g_{|U_{\alpha, \sigma}} - [\log | \sigma^*(f_{\alpha})|^2]}, (s_{|D})^{-1})\\
& = & (\sO(U_{\alpha}, f_{\alpha}),h_{\alpha, \sigma},  (s_{|D})^{-1})\\ 
& =^2& (\sL, ||\cdot ||, \phi),
\end{eqnarray*}
where the section $s$ in $=^1$ of the line bundle $\sO(U_{\alpha}, f_{\alpha})$  is given by the rational function $f_{\alpha}$ and the equality $=^2$ follows from the above local descriptions of $\sL$, $||\cdot ||$ and $\phi$. In particular, we have 
\begin{equation} \label{eqn:Chern-Iso*-5.1}
\cyc_{X|D} \circ \wh{c}_1 = \id \colon \wh{\Pic}(X, D) \to \wh{\Pic}(X, D). 
\end{equation}
We leave the identity $\wh{c}_1 \circ \cyc_{X|D}=\id$ as an easy exercise for the reader.
\end{proof}

\bibliographystyle{amsalpha}

\begin{thebibliography}{[GHV72]}

\bibitem[BE03]{BE03}
S. Bloch, H. Esnault, {\sl An additive version of higher Chow groups}, Annales
Scientifiques de l’\'Ecole Normale Su\'perieure, Volume {\bf 36}, (Elsevier, 2003), 463--477.\

\bibitem[BK18]{BAK18}  F. Binda, A. Krishna, {\sl Zero cycles with modulus and zero cycles on singular varieties\/}, Compositio Mathematica, {\bf 154}, (2018), 120--187. \

\bibitem[BS19]{Binda-Saito}  F. Binda, S. Saito, {\sl Relative cycles with moduli and
regulator maps\/}, J. Math. Inst. Jussieu, {\bf 18}, (2019), 1233--1293. \


\bibitem[GS90]{GS90} H. Gillet, C. Soule, {\sl Arithmetic intersection theory\/}, Publications
Math. IHES {\bf 72}, (1990), 94--174.\



\bibitem[GS90-2]{GS90-2} H. Gillet, C. Soule, {\sl Characteristic classes for algebraic vector bundles with Hermitian metrics\/}, I, II, Ann. of Math., {\bf 131}, (1990), 163--203 and 205-238.\

\bibitem[GKR20]{GKR20} R. Gupta, A. Krishna, J. Rathore, {\sl Tame class field theory over local fields\/}, arXiv:2209.02953v1 [math.AG]. \

\bibitem[GK22]{GK22} R. Gupta, A. Krishna, {\sl Idele class groups with modulus\/}, Advances in Mathematics, {\bf 404}, (2022) Part A.
\

\bibitem[GHV72]{G-H-V} W. Greub, S.Halperin, R. Vanstone, {\sl Connections, curvature and cohomology\/}, {\bf 1}, New York, Academic Press, (1972).\

\bibitem[GH78]{GH78} P. Griffiths, J. Harris, {\sl  Principles of algebraic geometry\/}, John Wiley and Sons, (1978).\

\bibitem[Ful84]{Fulton} W. Fulton, {\sl Intersection theory\/}, vol. 2 of Ergebnisse der Mathematik und ihrer Grenzgebiete (3) [Results in Mathematics
and Related Areas (3)], Springer-Verlag, Berlin, (1984).\



\bibitem[KS16]{Kerz-Saito} M. Kerz, S. Saito, {\sl Chow group of 0-cycles with 
modulus and higher-dimensional class field theory\/}, Duke Math. J., 
{\bf 165}, (2016), 2811--2897. \

%\bibitem[JK83]{King} James~R.~King, {\sl Log Complexes of Current and Functorial Properties of the Abel-Jacobi Map\/}, Duke Math. J., 
%{\bf 50}, (1983), 1--53. \

\bibitem[KL08]{KL-08} A. Krishna, M. Levine, {\sl Additive higher Chow groups of schemes\/}, J.
Reine Angew. Math. {\bf 619} (2008), 75–140. \

\bibitem[KP12]{KP12} A. Krishna, J. Park, {\sl Moving lemma for additive higher Chow groups\/}, Algebra Number Theory, {\bf 6}, (2012), 293–326.\

\bibitem[KP17]{KP17} A. Krishna, J. Park, {\sl A module structure and a vanishing theorem for cycles with modulus\/}, Math. Res. Lett., {\bf 24}, (2017), no. 4, 1147--1176.\

\bibitem[Lan91]{Lan91} S. Landsburg, {\sl Relative Chow groups\/}, Illinois J. Math., {\bf 35}, (1991), 618--641.\

\bibitem[Lee03]{Lee03} J. Lee, {\sl Introduction to Smooth Manifolds\/}, Springer New York, (2003).\

\bibitem[Mor14]{Moriwaki} A. Moriwaki, {\sl Arakelov Geometry\/}, Translations of Mathematical monographs, Vol {\bf 244}, (2014). \

\bibitem[Park09]{Park09} J. Park, {\sl Regulators on additive higher Chow groups}, Amer. J. Math. {\bf 131}:1 (2009), 257–276. \

\bibitem[StP]{StackP} Stacks project authors, {\sl The Stacks project\/}, \url{https://stacks.math.columbia.edu}, (2024).\

\bibitem[SV96]{Suslin-Voe} A. Suslin, V. Voevodky, {\sl Singular homology of abstract algebraic
varieties\/}, Invent. math.  {\bf 123}, (1996), 61--94.\



\bibitem[SABK]{Lectures} C. Soulé, D. Abramovich, J. Burnol, J. Kramer, {\sl Lectures on Arakelov Geometry\/}, 
 Cambridge Studies in Advanced Mathematics {\bf 33}, (1992)\

 \bibitem[Rob70]{Roberts} J. Roberts, {\sl Chow's moving lemma\/}, Algebraic Geometry (Oslo 1970) Proc. Fifth Nordic Summer School in Math. (Wolters-Noordhoff, Groningen, 1972), 89--96.


 
\end{thebibliography}

\end{document}